\documentclass[11pt]{article}

\usepackage[margin=1in]{geometry}

\usepackage{lipsum}

\usepackage[utf8]{inputenc} 
\usepackage[T1]{fontenc}    
\usepackage{tgtermes}

\usepackage{url}            
\usepackage{booktabs}       
\usepackage{amsfonts}       
\usepackage{nicefrac}       
\usepackage{microtype}      

\usepackage[english]{babel}

\usepackage{amsmath}
\usepackage{enumitem}
\usepackage{amssymb}
\usepackage{graphicx}
\usepackage{bm}
\usepackage{theorem}
\usepackage[colorlinks=true, allcolors=blue]{hyperref}

\newcommand{\algname}{\texttt{F$^2$SA}}
\newcommand{\algnametwo}{\texttt{F$^3$SA}}

\usepackage{mathtools}
\newenvironment{proof}{\paragraph{\it Proof.}}{\hfill$\square$}

\usepackage[title]{appendix}
\usepackage{comment}
\usepackage{amsfonts}
\usepackage{dsfont}
\usepackage{hyperref}
\usepackage{lipsum}
\usepackage{dsfont,subcaption,float}
\usepackage{algorithm}
\usepackage{algorithmic}

\usepackage{thm-restate}
\usepackage[size=small,labelfont=bf]{caption}

\usepackage{xcolor}

\theoremstyle{plain}
\newtheorem{theorem}{Theorem}[section]

\newtheorem{proposition}[theorem]{Proposition}
\newtheorem{lemma}[theorem]{Lemma}
\newtheorem{corollary}[theorem]{Corollary}
\theoremstyle{definition}
\newtheorem{definition}[theorem]{Definition}
\newtheorem{assumption}{Assumption}
\theoremstyle{remark}

\newcommand{\poly}{\mathrm{poly}}

\newcommand{\mF}{\mathcal{F}}	
\newcommand{\mD}{\mathcal{D}}	
\newcommand{\mL}{\mathcal{L}}	

\newcommand{\vdot}[2]{\langle #1, #2 \rangle}

\newcommand{\Exs}{\mathbb{E}}

\newcommand{\tssum}{\textstyle \sum}

\newcommand{\grad}{\nabla}

\allowdisplaybreaks

\newif\ifdraft
\drafttrue  

\newcommand{\blue}[1]{\textcolor{blue}{#1}}

\newcommand{\dkcomment}[1]{\ifdraft {\bf{{\blue{{Dohyun --- #1}}}}}\else\fi}

\newcommand{\I}{\mathcal{I}}
\newcommand{\J}{\mathcal{J}}

\title{\bf{\LARGE{A Fully First-Order Method for Stochastic Bilevel Optimization}}}

\usepackage{authblk}
\author[1]{Jeongyeol Kwon}
\author[1]{Dohyun Kwon}
\author[1]{Stephen Wright} 
\author[1]{Robert Nowak}
\affil[1]{Wisconsin Institute for Discovery, UW-Madison}

\begin{document}
\maketitle

\begin{abstract}
    We consider stochastic unconstrained bilevel optimization problems when only the first-order gradient oracles are available. 
    While numerous optimization methods have been proposed for tackling bilevel problems, existing methods either tend to require possibly expensive calculations involving Hessians of lower-level objectives, or lack rigorous finite-time performance guarantees. 
    In this work, we propose a Fully First-order Stochastic Approximation (\algname) method, 
    and study its non-asymptotic convergence properties. Specifically, we show that \algname~converges to an $\epsilon$-stationary solution of the bilevel problem after $\epsilon^{-7/2}, \epsilon^{-5/2}$, or $\epsilon^{-3/2}$ iterations (each iteration using $O(1)$ samples)  when stochastic noises are in both level objectives, only in the upper-level objective, or not present (deterministic settings), respectively. 
    We further show that if we employ momentum-assisted gradient estimators, the iteration complexities can be improved to $\epsilon^{-5/2}, \epsilon^{-4/2}$, or $\epsilon^{-3/2}$, respectively. 
    We demonstrate the superior practical performance of the proposed method over existing second-order based approaches on MNIST data-hypercleaning experiments.
\end{abstract}

\section{Introduction}

Bilevel optimization \cite{colson2007overview} arises in many important applications that have two-level hierarchical structures, including meta-learning \cite{rajeswaran2019meta}, hyper-parameter optimization \cite{franceschi2018bilevel, bao2021stability}, model selection \cite{kunapuli2008bilevel, giovannelli2021bilevel}, adversarial networks \cite{goodfellow2020generative, gidel2018variational}, game theory \cite{stackelberg1952theory} and reinforcement learning \cite{konda1999actor, sutton2018reinforcement}. Bilevel optimization can be generally formulated as the following minimization problem:
\begin{align}
    &\min_{x \in X} \quad F(x) := f(x,y^*(x)) \nonumber \\
    &\text{s.t.} \quad y^*(x) \in \arg \min_{y \in \mathbb{R}^{d_y}} g(x,y), \label{problem:bilevel} \tag{\textbf{P}}
\end{align}
where $f$ and $g$ are continuously differentiable functions and $X \subseteq \mathbb{R}^{d_x}$ is a convex set. 
The outer objective $F$ depends on $x$ both directly and also indirectly via $y^*(x)$, which is a solution of the  lower-level problem of minimizing another function $g$, which is parametrized by $x$. 
Throughout the paper, we assume that $X = \mathbb{R}^{d_x}$ (that is, there are no explicit constraints on $x$) and  that $g(x,y)$ is strongly convex in $y$, so that $y^*(x)$ is uniquely well-defined for all $x \in X$. 


Among various approaches to \eqref{problem:bilevel}, iterative procedures have been predominant due to their simplicity and potential scalability in large-scale applications. Initiated by \cite{ghadimi2018approximation}, a flurry of recent works study efficient iterative procedures and their finite-time performance for solving \eqref{problem:bilevel}, see {\it e.g.,}  \cite{chen2021closing, hong2020two, khanduri2021near, chen2022single, dagreou2022framework, guo2021randomized, sow2022constrained, ji2021bilevel, yang2021provably}. The underlying idea is based on an algorithm of (stochastic) gradient descent type, applied to $F$, that is,
\begin{align*}
    x_{k+1} = x_k - \alpha_k \grad F(x_k),
\end{align*}
with some appropriate step-sizes $\{\alpha_k\}$. 
Direct application of this approach requires us to  compute or estimate the so-called hyper-gradient of $F$ at $x$, which is
\begin{align} \label{eq:gradF}
    \grad F(x) = &\grad_x f (x, y^*(x)) - 
    \grad_{xy}^2 g(x, y^*(x))  
    \grad_{yy}^2 g(x, y^*(x))^{-1} \grad_y f(x, y^*(x)). 
\end{align}
There are two major obstacles in computing \eqref{eq:gradF}. 
The first obstacle is that for every given $x$, we need to search for the optimal solution $y^*(x)$ of the lower problem, which results in updating the lower variable $y$ multiple times before updating $x$. 
To tackle this issue, several ideas have been proposed in \cite{ghadimi2018approximation, hong2020two, chen2021closing} to effectively track $y^*(x)$ without waiting for too many inner iterations before updating $x$ (we discuss this further in Section \ref{section:related_work}).
Following in the spirit of this approach, we show that a single-loop style algorithm can still be implemented using only first-order gradient estimators. 

The second obstacle, which is the main focus of this work, centers around the presence of second-order derivatives of $g$ in \eqref{eq:gradF}.
Existing approaches mostly require an explicit extraction of second-order information from $g$ with a major focus on estimating the Jacobian and inverse Hessian efficiently with stochastic noises \cite{ji2021bilevel, chen2022single, dagreou2022framework}. 
We are particularly interested  in regimes in which such operations are costly and prohibitive \cite{mehra2021penalty, giovannelli2021bilevel}. 
Some existing works avoid the second-order computation and only use the first-order information of both upper and lower objectives; see \cite{giovannelli2021bilevel, sow2022constrained, liu2021towards, ye2022bome}. 
These works either lack a complete finite-time analysis \cite{giovannelli2021bilevel, liu2021towards} or are applicable only to deterministic functions \cite{ye2022bome, sow2022constrained}. 

Our goal in this paper is to study a {\it fully} first-order approach for stochastic bilevel optimization. 
We propose a gradient-based approach that avoids  the estimation of Jacobian and Hessian of $g$, and finds an $\epsilon$-stationary solution of $F$ using only first-order gradients of $f$ and $g$. 
Further,
the number of inner iterations  remains constant throughout all outer iterations of our algorithm. 
We provide a finite-time analysis of our method with explicit convergence rates. 
To our best knowledge, this work is the first to establish non-asymptotic convergence guarantees for stochastic bilevel optimization using only first-order gradient oracles.

\subsection{Overview of Main Results}
The starting point of our approach is to convert \eqref{problem:bilevel} to an equivalent constrained single-level version:
\begin{equation}
    \min_{x \in X, \ y \in \mathbb{R}^{d_y}} \quad f(x,y) \ \quad \text{s.t.} \ \ \quad g(x,y) - g^*(x) \le 0, \label{problem:penalty_bilevel} \tag{\textbf{P'}}
\end{equation}
where $g^*(x) := g(x, y^*(x))$. 
The Lagrangian $\mL_{\lambda}$ for \eqref{problem:penalty_bilevel} with multiplier $\lambda > 0$ is
\begin{align*}
    \mL_\lambda(x,y) := f(x,y) + \lambda (g(x,y) - g^* (x)).
\end{align*}
We can minimize $\mL_{\lambda}$ for a given $\lambda$ by, for example, running (stochastic) gradient descent. 
As noted in \cite{ye2022bome}, the gradient of $\mL_{\lambda}$ can be computed only with gradients of $f$ and $g$, and thus the entire procedure can be implemented using only with first-order derivatives. 
In fact, such a reformulation has been attempted in several recent works (e.g., \cite{liu2021value, sow2022constrained, ye2022bome}). 
However, the challenge in handling the constrained version \eqref{problem:penalty_bilevel} is to find an appropriate value of the multiplier $\lambda$. 
Unfortunately, the desired solution $x^* = \arg\min_x F(x)$ can only be obtained at $\lambda = \infty$
(in fact, the gradient of the constraint in \eqref{problem:penalty_bilevel} is zero with respect to $(x,y)$ at the minimizer, and thus the so-called {\it constraint qualifications} \cite{wright1999numerical} fail to hold). 
However, with $\lambda=\infty$, $\mL_{\lambda}(x,y)$ has unbounded smoothness which prevents us from employing gradient-descent style approaches. 
For these reasons, none of the previously proposed algorithms  can obtain a consistent estimator for the original problem $\min_x F(x)$ without access to second derivatives of $g$.

Nonetheless, we find that \eqref{problem:penalty_bilevel} is the key to deriving a consistent estimator that converges to an $\epsilon$-stationary point of $F$ in finite time {\it without} access to second derivatives. 
The main idea is to start with an initial value $\lambda=\lambda_0>0$ and gradually increase it on subsequent iterations: At iteration $k$,  $\lambda_k = O(k^b)$ for some $b \in (0, 1]$. 
The success of this approach depends crucially on the growth rate captured by the parameter $b$. 
On one hand, fast growth of $\lambda_k$  removes the bias quickly.
On the other hand, fast growth of $\lambda_k$ forces a fast decay of step-sizes due to the growing nonsmoothness  of $\mL_{\lambda_k}$, which slows down the overall convergence.

Our main technical contribution is to characterize an explicit growth rate of $\lambda_k$ that optimizes the trade-off between bias and step-sizes, and to provide a non-asymptotic convergence guarantee with explicit rates for the proposed algorithm. 
\begin{itemize}
    \item We propose a fully first-order method, \algname, for stochastic bilevel optimization. \algname~is a single-loop style algorithm: For every outer variable update we only update inner variables a constant number of times. 
    \item We characterize explicit convergence rates of \algname~in different stochastic regimes. 
    It converges to an $\epsilon$-stationary-point of \eqref{problem:bilevel} after $\tilde{O} (\epsilon^{-3.5})$, $\tilde{O}(\epsilon^{-2.5})$, or $\tilde{O}(\epsilon^{-1.5})$ iterations if both $\grad f$ and $\grad g$ contain stochastic noise, if only access to $\grad f$ is noisy, or if we are in deterministic settings, respectively. 
    These complexities can be improved to $\tilde{O}(\epsilon^{-2.5}), \tilde{O}(\epsilon^{-2})$, or $\tilde{O}(\epsilon^{-1.5})$, respectively, if momentum or variance-reduction techniques are employed. 
    The crux of the analysis is to understand the effect of the value of multipliers $\lambda_k$ on step-sizes, noise variances, and bias. 
    \item We demonstrate the proposed algorithm on a data hyper-cleaning task for MNIST. 
    Even though our theoretical guarantees are not better than existing methods that use second-order information, we illustrate that \algname~can even outperform such methods in practice. 
\end{itemize}

\subsection{Related Work}
\label{section:related_work}

Bilevel optimization has a long and rich history since its first introduction in \cite{bracken1973mathematical}. A number of algorithms have been proposed for bilevel optimization. Classical results include approximation descent \cite{vicente1994descent} and penalty function method \cite{ishizuka1992double, anandalingam1990solution, white1993penalty} for instance; see \cite{colson2007overview} for a comprehensive overview. These results often deal with a several special case of bilevel-optimization and only provide asymptotic convergence. Note that the penalty function methods in \cite{ishizuka1992double, anandalingam1990solution, white1993penalty} are specifically developed for the subclass of their own interest, and cannot be applied to general non-convex objectives $F$.

More recent results on bilevel optimization focuses on non-asymptotic analysis when the lower-level problem is strongly-convex in $y$, since in this case $\grad F(x)$ can be expressed in the closed-form \eqref{eq:gradF} using the implicit function theorem \cite{couellan2015bi, couellan2016convergence}. As mentioned earlier, there are two major challenges in this setting: (i) finding a good approximation of $y^*(x)$, and (ii) the evaluation of Jacobian and Hessian inverse of $g$. The work in \cite{ghadimi2018approximation} establishes the first non-asymptotic analysis of a double-loop algorithm, where in the inner problem we find an approximate solution of $y^*(x)$ given $x$, and use it to evaluate an approximation of $\grad F(x)$. Furthermore, \cite{ghadimi2018approximation} uses the Neuman series approximation to estimate the Hessian inverse when we only have access to the stochastic oracles (of second-order derivatives).

The paper \cite{ghadimi2018approximation} was followed by a flurry of work that improved their result in numerous ways. 
For instance, \cite{hong2020two, chen2021closing, chen2022single, ji2021bilevel} develop a single-loop style update by properly choosing two step-sizes for the inner and outer iterations, along with the improved sample complexity, {\it i.e.,} the total number of accesses to first and second-order stochastic oracles. 
The overall convergence rate is further optimized by using variance-reduction and momentum techniques \cite{khanduri2021near, dagreou2022framework, guo2021randomized, yang2021provably, huang2021biadam}. 
We do not aim to compete with the convergence rates obtained from this line of work, since all of these method have access to second-order derivatives, even though some computational cost might be saved if good automatic differentiation packages \cite{margossian2019review} are available. 
Rather, we avoid the needs for second-order information altogether, allowing a simple algorithm with low per-iteration complexity for large scale applications.

The results most closely related to ours can be found in \cite{ye2022bome, sow2022constrained}. 
\cite{sow2022constrained} considers a primal-dual approach for \eqref{problem:penalty_bilevel}, but their main focus is to get a biased solution when $g$ is only convex (not strongly convex), so the lower-level problem may have multiple solutions. 
Their analysis is restricted to the case in which the overall Lagrangian is strongly-convex in $x$ (which is not usually guaranteed) and they do not provide any guarantees in terms of the true objective $F$. 
More recent work in \cite{ye2022bome} is the closest to ours, but they only consider  deterministic gradient oracles, and do not provide convergence guarantees in terms of $F$. 
Moreover, they  prove a convergence guarantee of $O(k^{-1/4})$, whereas we show an improved guarantee of $\tilde{O}(k^{-2/3})$ in the deterministic case.

We note that there is a line of work that studies a different version of the bilevel problem which has no coupling between two variables $x$ and $y$ ({\it e.g.,} see \cite{ferris1991finite, solodov2007explicit, jiang2022conditional}). In \cite{amini2019iterative, amini2019iterative2}, the Lagrangian formulation is exploited with iteratively increasing multiplier. Note that the nature of single-variable bilevel formulation is different from \eqref{problem:bilevel} as the former is only interesting when the lower-level problem allows a multiple (convex) solution set. To our best knowledge, the idea of iteratively increasing $\lambda_k$ with its non-asymptotic guarantee is new in the context of solving \eqref{problem:bilevel}, and has the merit of avoiding (possibly) expensive second-order computation.

\section{Preliminaries}
We state several assumptions on \eqref{problem:bilevel} to specify the problem class of interest. We assume that optimal value of the outer-level objective is bounded below $F^* := \arg \min_x F(x) > -\infty$. We consider \eqref{problem:bilevel} with the following assumptions on objective functions:
\begin{assumption} 
    \label{assumption:nice_functions}
    The following holds for objective functions $f$ and $g$:
    \begin{enumerate}
        \item[1.] $f$ is continuously differentiable and $l_{f,1}$-smooth jointly in $(x,y)$.
        \item[2.]  $g$ is continuously differentiable and $l_{g,1}$-smooth jointly in $(x,y)$.
        \item[3.] For every $\bar{x} \in X$, $\|\grad_y f(\bar{x}, y)\|$ is bounded by $l_{f,0}$ for all $y$.
    \end{enumerate}
\end{assumption}
Assumption \ref{assumption:nice_functions} assumes the smoothness of both objective functions and boundedness of $\grad_y f$. This has been a standard assumption in bilevel optimization \cite{ghadimi2018approximation}. In this work, we focus on well-conditioned bilevel optimization problems, {\it i.e.,} when $F(x)$ is well-defined, continuous and smooth. The following assumption has been the standard sufficient condition for well-conditioned bilevel problems \cite{ghadimi2018approximation}:
\begin{assumption}
    \label{assumption:extra_nice_g}
    The following holds for the lower-level objective $g$:
    \begin{enumerate}
        \item[1.] For every $\bar{x} \in X$, $g(\bar{x}, y)$ is $\mu_g$ strongly-convex in $y$ for some $\mu_g > 0$.
        \item[2.] $g$ is two-times continuously differentiable, and $\grad^2 g$ is $l_{g,2}$-Lipschitz jointly in $(x,y)$.
    \end{enumerate}
\end{assumption}
We assume that we can access first-order information of objective functions only through stochastic gradient oracles:
\begin{assumption}
    \label{assumption:gradient_variance}
    We access the gradients of objective functions via unbiased estimators $\nabla f(x,y;\zeta), \nabla g(x,y;\phi)$ where:
    \begin{align*}
        & \Exs[\nabla f(x,y;\zeta)] = \nabla f(x, y), \\
        & \Exs[\nabla g(x,y;\phi)] = \nabla g(x, y),
    \end{align*}
    and the variances of stochastic gradient estimators are bounded:
    \begin{align*}
        & \Exs[\|\nabla f(x,y; \zeta) - \nabla f(x, y)\|^2] \le \sigma_f^2, \\
        & \Exs[\|\nabla g(x,y; \phi) - \nabla g(x, y)\|^2] \le \sigma_g^2.
    \end{align*} 
\end{assumption}
Throughout the paper, we assume that Assumptions \ref{assumption:nice_functions}-\ref{assumption:gradient_variance} hold unless specified otherwise. We use the following definition as the optimality criteria for solving \eqref{problem:bilevel}:
\begin{definition}[$\epsilon$-stationary point]
    A point $x$ is called $\epsilon$-stationary if $\|\grad F(x)\|^2 \le \epsilon$. A stochastic algorithm is said to achieve an $\epsilon$-stationary point in $K$ iterations if $\Exs[\|\grad F(x_K)\|^2]\le \epsilon$ where the expectation is over the stochasticity of the algorithm.
\end{definition}

\paragraph{Notation.} We use $O_{\texttt{P}} (\cdot)$ when we state the order of constants which depends on instance-dependent parameters ({\it e.g.,} Lipschitz, strong-convexity, smoothness parameters). We say $a_k \asymp b_k$ if $a_k$ and $b_k$ decreases (or increases) in the same rate as $k \rightarrow \infty$, {\it i.e.,} $\lim_{k\rightarrow \infty} a_k/b_k = \Theta(1)$. Throughout the paper, $\|\cdot\|$ denotes the Euclidean norm on finite dimensional space.

\section{Algorithm}
\label{section:algorithm}

In this section, we develop an algorithm that converges to a stationary point of the bilevel problem ({\it i.e.,} a stationary point of $F(x) = f(x,y^*(x))$) and makes use only of gradients of $f$ and $g$. 
Recall the  equivalent formulation \eqref{problem:penalty_bilevel}. 
To see how we can avoid second-order derivatives, we observe the gradient of $\grad \mL_{\lambda}$:
\begin{align*}
    \grad_x \mL_{\lambda}(x,y) &= \grad_x f(x,y) + \lambda (\grad_x g(x,y) - \grad g^*(x)), \\
    \grad_y \mL_{\lambda}(x,y) &= \grad_y f(x,y) + \lambda \grad_y g(x,y).
\end{align*}
Note that
\begin{align*}
    \grad g^*(x) &= \grad_x g(x,y^*(x)) + \grad_x y^*(x) \grad_y g(x, y^*(x)) = \grad_x g(x,y^*(x)),
\end{align*}
due to the optimality condition for $g$ at $y^*(x)$. 
Thus, we could consider optimizing $\mL_{\lambda} (x,y)$ by introducing an auxiliary variable $z$ that chases $y^*(x)$, and setting up an alternative bilevel formulation \eqref{problem:bilevel} with outer-level objective $\mL_{\lambda}(x', z)$, outer variable $x' = (x,y)$, and inner variable $z$. 
However, such an approach settles in a different landscape from that of  $F(x)$, resulting in a bias. 
The question is how tightly we can control this bias without compromising too much smoothness of the alternative function $\mL_{\lambda}$, which affects the overall step-size design and noise variance. 


To control the bias, we need a better understanding of how the functions $\mL_{\lambda}$ and $F(x)$ are related. 
Let us introduce an auxiliary function $\mL_{\lambda}^*$ defined as:
\begin{align*}
    \mL_{\lambda}^*(x) := \min_y \mL_\lambda (x,y).
\end{align*}
Note that if $\lambda > 2l_{f,1} / \mu_g$, then for every $\bar{x} \in X$, $\mL_{\lambda}(\bar{x},y)$ is at least $(\lambda \mu_g / 2)$ strongly-convex in $y$, and therefore its minimizer $y_{\lambda}^*(x)$ is uniquely well-defined:
\begin{align}
\label{eq:yls}
    y_{\lambda}^*(x) := \arg \min_y \mL_{\lambda} (x,y). 
\end{align}
Since $F(x) = \lim_{\lambda \rightarrow \infty} \mL_{\lambda}^*(x)$ for every $x \in X$, we could expect that $\mL_{\lambda}^*(x)$ is a well-defined proxy of $F(x)$ for sufficiently large $\lambda > 0$.
The following lemma confirms this intuition.
\begin{lemma}
    \label{lemma:l_star_lambda_approximate}
    For any $x \in X$ and $\lambda \ge 2l_{f,1} / \mu_g$, $\grad \mL_{\lambda}^*(x)$ is given by
    \begin{align*}
        \grad_x \mL_{\lambda} (x, y_{\lambda}^*(x)) &= \grad_x f(x,y_{\lambda}^*(x)) + \lambda (\grad_x g(x, y_{\lambda}^*(x)) - \grad_x g(x, y^*(x))).
    \end{align*}
    Furthermore,  we have
    \begin{align*}
        \| \grad F(x) - \grad \mL_\lambda^*(x) \| \le  C_\lambda / \lambda,
    \end{align*}
    where $C_\lambda := \frac{4 l_{f,0} l_{g,1}}{\mu_g^2} \left( l_{f,1} + \frac{2 l_{f,0} l_{g,2}}{\mu_g} \right)$. 
\end{lemma}
Importantly, $\grad \mL_{\lambda}^*(x)$ can be computed only with first-order derivatives of both $f$ and $g$. Thus any first-order method that finds a stationary point of $\mL_{\lambda}^*(x)$ approximately follows the trajectory of $x$ updated with the exact $\grad F(x)$, with a bias of $O(1/\lambda)$.

\begin{algorithm}[t]
    \caption{\algname}
    \label{algo:algo_name}
    
    {{\bf Input:} step sizes: $\{\alpha_k, \gamma_k\}$, multiplier difference sequence: $\{\delta_k\}$, inner-loop iteration count: $T$, step-size ratio: $\xi$, initializations: $\lambda_0, x_0, y_0, z_0$}
    \begin{algorithmic}[1]
        \FOR{$k = 0 ... K-1$}
            \STATE{$z_{k,0} \leftarrow z_{k}$, $y_{k,0} \leftarrow y_{k}$}
            \FOR{$t = 0 ... T-1$}
                \STATE{$z_{k,t+1} \leftarrow z_{k,t} - \gamma_k h_{gz}^{k,t}$}
                \STATE{$y_{k,t+1} \leftarrow y_{k,t} - \alpha_k (h_{fy}^{k,t} + \lambda_k h_{gy}^{k,t})$}
            \ENDFOR
            \STATE{$z_{k+1} \leftarrow z_{k,T}$, $y_{k+1} \leftarrow y_{k,T}$}
            \STATE{$x_{k+1} \leftarrow x_k - \xi \alpha_k (h_{fx}^k + \lambda_k (h_{gxy}^k - h_{gxz}^k))$}
            \STATE{$\lambda_{k+1} \leftarrow \lambda_k + \delta_k$}
        \ENDFOR
    \end{algorithmic}
\end{algorithm}

Our strategy is to use $\grad \mL_{\lambda}^*(x)$ as a proxy to $\grad F(x)$ for generating a sequence of iterates $\{x_k\}$. 
Accordingly, we introduce sequences $\{y_k\}$ and $\{z_k\}$ that approximate $y_{\lambda_k}^*(x_k)$ and $y^*(x_k)$, respectively. 
We gradually increase $\lambda_k$ with $k$, so that the bias in the sequence $\{x_k\}$ converges to 0.

Our \textbf{F}ully \textbf{F}irst-order \textbf{S}tochastic \textbf{A}pproximation (\algname) method is shown in Algorithm \ref{algo:algo_name}. 
We emphasize that the method works with \emph{stochastic} gradients that are independent unbiased estimators of gradients, {\it i.e.,} 
\begin{align*}
    &h_{gz}^{k,t} := \nabla_y g(x_k, z_{k,t}; \phi_{z}^{k,t}), \ h_{fy}^{k,t} := \nabla_y f(x_k, y_{k,t}; \zeta_{y}^{k,t}), \\
    &h_{gy}^{k,t} := \nabla_y g(x_k, y_{k,t}; \phi_{y}^{k,t}), h_{gxy}^k:= \nabla_{x} g(x_k, y_{k+1}; \phi_{xy}^k), \\
    & h_{fx}^k := \nabla_x f(x_k, y_{k+1}; \zeta_{x}^k), h_{gxz}^k := \nabla_x g(x_k, z_{k+1}; \phi_{xz}^k).
\end{align*}
The algorithm can set $T=1$ in conjunction with an appropriate choice of $\xi$, allowing a fully single-loop update for all variables.

\subsection{Step-Size Design Principle} 
\label{sec:step}

We describe how we design the step-sizes for Algorithm \ref{algo:algo_name} to  achieve convergence to a $\epsilon$-stationary point of $F$. 
Several conditions must be satisfied. 
As will be shown in the analysis, with $(\lambda_k \mu_g/2)$-strong convexity of $\mL_{\lambda_k}$ in $y$, one-step inner iteration of $y_{k,t}$ is a contraction mapping toward $y_{\lambda,k}^*$ with rate $1 - O(\mu_g \beta_k)$. Henceforth, we often use the notatino $\beta_k := \alpha_k\lambda_k$, which is the effective step-size for updating $y_k$. 
For simplicity, we denote $y_{\lambda,k}^* := y_{\lambda_k}^* (x_k)$ and $y_k^* := y^*(x_k)$.  

We now describe the specific rules. First, essentially the step size we use for $x_k$ is $\xi \alpha_k$, and thus, $\alpha_k = \Omega(1/k)$ (otherwise, $\sum_k \alpha_k < \infty$, and $x_k$ fails to converge). 
On the other hand, since $\beta_k=\alpha_k\lambda_k$ is the effective step size for updating $y_k$, we need $\beta_k < O(1/l_{g,1}) = O(1)$. 
Together, these observations imply that the maximum rate of growth for $\lambda_k$ cannot exceed $O(k)$. 

Second, note that $\|x_{k+1} - x_k\|$ is (roughly) proportional to
\begin{align*}
    \|\grad F(x_k)\| + C_\lambda/\lambda_k + \lambda_k \|y_k - y_{\lambda, k}^*\| + \lambda_k \|z_k - y_{k}^*\|.
\end{align*}
This rate is optimized when $\|y_{k} - y_{\lambda,k}^*\| \asymp \|z_k - y_k^*\| \asymp \lambda_k^{-2}$. Thus, the ideal growth rate for $\lambda_k$ is $\|y_{k} - y_{\lambda,k}^*\|^{1/2}$ or $\|z_k - y_k^*\|^{1/2}$. 
We will design the rate of convergence of $y_k$ and $z_k$ to be the same, {\it i.e.,} $\beta_k \asymp \gamma_k$. 
For instance, when we have stochastic noises in the gradient estimate of $g$, {\it i.e.,} $\sigma_g^2 > 0$, the expected convergence rate of $\|y_k - y_{\lambda_k}^*\|^2$ is $O(\beta_k)$, since the sequence is optimized for strongly convex functions. 
This suggests $\lambda_k \asymp \beta_k^{-1/4}$ as the ideal rate of growth for $\lambda_k$.

\begin{figure}[t]
    \centering
    \includegraphics[width=0.4\textwidth]{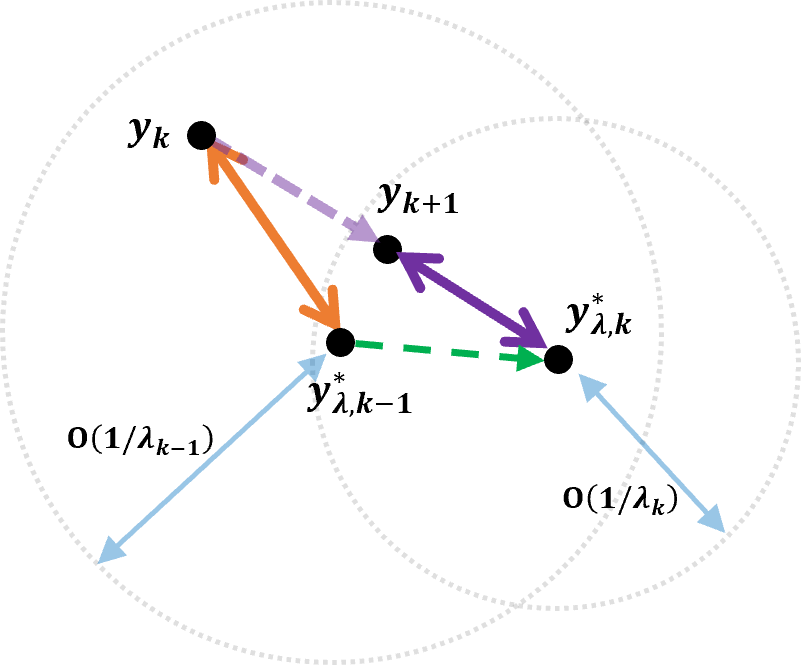}
    \caption{$y_k$ should move faster than $y_{\lambda_k}^*(x_k)$ moves, and stay within $O(1/\lambda_k)$-ball around $y_{\lambda_k}^*(x_k)$.}
    \label{fig:contraction_explain}
\end{figure}

The crux of Algorithm \ref{algo:algo_name} is how well $y_k$ (and $z_k$) can chase $y_{\lambda_k}^*(x_k)$ (resp. $y^*(x_k)$) when $x_k$ and $\lambda_k$ are changing  at every iteration. 
We characterize first how fast $y_\lambda^*(x)$ moves in relation to the movements of $\lambda$ and $x$.
\begin{lemma}
    \label{lemma:y_star_lagrangian_continuity}
    For any $x_1, x_2 \in X$ and for any $\lambda_2 \ge \lambda_1 \ge 2 l_{f,1} / \mu_g$, we have
    \begin{align*}
        \| y_{\lambda_1}^*(x_1) - y_{\lambda_2}^*(x_2)\| \le \frac{ 2(\lambda_2 - \lambda_1)}{\lambda_1\lambda_2} \frac{l_{f,0}}{\mu_g} + l_{\lambda,0} \|x_2 - x_1\|,
    \end{align*}
    for some $l_{\lambda,0} \le 3l_{g,1} / \mu_g$.
\end{lemma}
For Algorithm \ref{algo:algo_name} to converge to a desired point, $y_k$ should move sufficiently fast toward the current target $y_{\lambda,k}^*$ every iteration, dominating the movement of target $y_{\lambda, k}^*$ that results from updates to $x_k$ and $\lambda_k$ (see Figure~\ref{fig:contraction_explain}). 
At a minimum, the following condition should hold (in expectation):
\begin{align*}
    \|y_{k+1} - y_{\lambda, k}^*\| < \| y_k - y_{\lambda, k-1}^* \|.
\end{align*}

Since $\|y_{k+1} - y_{\lambda, k}^*\|^2$ can be bounded with $T$-steps of $1 - O(\mu_g \beta_k)$ contractions, starting from $y_k$, we require 
\begin{align*}
    \left(1 - O(T \mu_g \beta_k) \right) \|y_k - y_{\lambda, k}^*\|^2 < \|y_k - y_{\lambda, k-1}^*\|^2.
\end{align*}
Now, applying the bound in Lemma \ref{lemma:y_star_lagrangian_continuity}, the minimal condition is given by:
\begin{align*}
    \|y_{\lambda, k-1}^* - y_{\lambda,k}^*\| &\le (l_{f,0} / \mu_g) \cdot (\delta_k / \lambda_k^2) + l_{\lambda,0} \|x_{k} - x_{k-1}\| \le  T \mu_g \beta_k \| y_{k} - y_{\lambda, k-1}^* \|.
\end{align*}
Note that $\|y_{k+1} - y_{\lambda, k}^*\|$ should decay faster than $\lambda_k^{-1}$. Otherwise, the bias in updating $x_k$ using $y_k$ (to estimate $\grad \mL_{\lambda_k}^*$) is larger than $\lambda_k \|y_{k+1} - y_{\lambda,k}^*\|$, and this amount might blow up. Also, it can be easily seen that $\|x_{k} - x_{k-1}\| = \Omega(\xi \beta_k \|y_{k} - y_{\lambda,k-1}^*\|)$. 
We can thus derive two simple conditions:
\begin{align*}
    \frac{\delta_k}{\lambda_k} \le O_{\texttt{P}} (1) \cdot \beta_k, \ \frac{\xi}{T} < O_{\texttt{P}} (1),
\end{align*}
where $O_{\texttt{P}}(1)$ are instance-dependent constants. If $\lambda_k$ grows in some polynomial rate, then $\delta_k / \lambda_k = O(1/k)$ and the first condition is satisfied provided that $\beta_k = \Omega(1/k)$. 
The second condition indicates the number of inner iterations $T$ required for each outer iteration. 
We can set $T=1$ (thus making the algorithm single-loop) by setting $\xi$ sufficiently small.
Alternatively, we can set $\xi = 1$ choose $T>1$ to depend on some instance-specific parameters.

\subsection{Extension: Integrating Momentum}
\begin{algorithm}[t]
    \caption{\algnametwo}
    \label{algo:algo_name2}
    
    {{\bf Input:} step sizes: $\{\alpha_k, \gamma_k\}$, multiplier difference sequence: $\{\delta_k\}$, momentum-weight sequence $\{\eta_k\}$, step-size ratio: $\xi$, initialization: $\lambda_0, x_0, y_0, z_0$}
    \begin{algorithmic}[1]
        \FOR{$k = 0 ... K-1$}
            \STATE{$z_{k+1} \leftarrow z_{k} - \gamma_k \tilde{h}_{gz}^{k}$}
            \STATE{$y_{k+1} \leftarrow y_{k} - \alpha_k (\tilde{h}_{fy}^{k} + \lambda_k \tilde{h}_{gy}^{k})$}
            \STATE{$x_{k+1} \leftarrow x_k - \xi \alpha_k (\tilde{h}_{fx}^k + \lambda_k (\tilde{h}_{gxy}^k - \tilde{h}_{gxz}^k))$}
            \STATE{$\lambda_{k+1} \leftarrow \lambda_k + \delta_k$}
        \ENDFOR
    \end{algorithmic}
\end{algorithm}

Given the simple structure of Algorithm \ref{algo:algo_name}, we can integrate variance-reduction techniques to improve the overall convergence rates. 
One relevant technique is the momentum-assisting technique of \cite{khanduri2021near} for stochastic bilevel optimization. 
To simplify the presentation, we consider a fully single-loop variant by setting $T=1$. 

To apply the momentum technique, we only need to replace the simple unbiased gradient estimators $h$ with momentum-assisted gradient estimators $\tilde{h}$. 
For instance, $\tilde{h}_z^k$ can be defined with a proper momentum weight sequence $\eta_k \in (0,1]$ as follows:
\begin{align*}
    \tilde{h}_{z}^{k} := &\nabla_y g(x_k, z_{k}; \phi_{z}^{k}) + (1 - \eta_{k}) \left( \tilde{h}_{z}^{k-1} - \grad_y g(x_{k-1}, z_{k-1}; \phi_{z}^{k}) \right).
\end{align*}
Other quantities $\tilde{h}_{fy}$, $\tilde{h}_{gy}$, $\tilde{h}_{fx}$, $\tilde{h}_{gxy}$, $\tilde{h}_{gxz}$ are defined similarly, with the same momentum-weight sequence. 
We defer the full description of those quantities to Appendix~\ref{appendix:momentum_method}. 
The version of our algorithm that incorporates momentum is called \textbf{F}aster \textbf{F}ully \textbf{F}irst-order \textbf{S}tochastic \textbf{A}pproximation (\algnametwo); it is described in Algorithm~\ref{algo:algo_name2}, where we simply replace $h$ with $\tilde{h}$. 
Note that we have additional moment-weight parameters $\{\eta_{k}\}$.

\section{Main Results}
\label{section:analysis}

In this section, we provide non-asymptotic convergence guarantees of the proposed algorithms. For Algorithm~\ref{algo:algo_name}, we prove in  Theorem~\ref{theorem:general_nonconvex} that the weighted sum of $\|\grad F(x_k)\|^2$ in expectation is bounded from above. Choosing suitable step sizes, the estimate guarantees the convergence rate, which depends on the presence of stochastic noises in Corollaries \ref{corollary:non_convex_F}. The similar results with better convergence rates and weaker assumptions hold true for Algorithm~\ref{algo:algo_name2}, as shown in Theorem~\ref{theorem:general_momentum}
and Corollary \ref{corollary:non_convex_momentum}.








\subsection{Main Result for Algorithm \ref{algo:algo_name}}

In this section we provide non-asymptotic convergence guarantees of the proposed algorithms. 
For Algorithm~\ref{algo:algo_name}, we prove in  Theorem~\ref{theorem:general_nonconvex} that the weighted sum of $\|\grad F(x_k)\|^2$ in expectation is bounded from above. 
By choosing suitable step sizes, the estimate yields a convergence rate.
Dependence on stochastic noises is explicated in Corollaries~\ref{corollary:non_convex_F}. 
Similar results with better convergence rates and weaker assumptions are proved for Algorithm~\ref{algo:algo_name2}; see Theorem~\ref{theorem:general_momentum}
and Corollary~\ref{corollary:non_convex_momentum}.








\subsection{Main Result for Algorithm \ref{algo:algo_name}}

Two mild assumptions are required for exploiting the smoothness of $y_{\lambda}^*(x)$.
\begin{assumption}
    \label{assumption:bounded_grad_x}
    THe gradient with respect to $x$ is bounded for functions $f$ and $g$:
    \begin{enumerate}
        \item[1.] For every $\bar{y}$, $\|\grad_x f(x, \bar{y})\| \le l_{f,0}$ for all $x \in X$.
        \item[2.] For every $\bar{y}$, $\|\grad_x g(x, \bar{y})\| \le l_{g,0}$ for all $x \in X$.
    \end{enumerate}
\end{assumption}
\begin{assumption}
    \label{assumption:extra_smooth_f}
    $f$ is two-times continuously differentiable, and $\grad^2 f$ is $l_{f,2}$-Lipschitz in $(x,y)$.
\end{assumption}
The smoothness of $y^*_\lambda(x)$ is used to keep the number of effective inner iterations constant throughout all outer-iterations, as in \cite{chen2021closing}. 


Before we state our convergence result, let us define some additional notation. We denote the second-moment bound of the $x$ update, $x_{k+1} - x_k$, as $M := \max(l_{f,0}^2 + \sigma_f^2, l_{g,0}^2 + \sigma_g^2)$. We also denote $l_{*,0} = \max(1, l_{\lambda_0,0})$ 
and $l_{*,1} = l_{\lambda_0,1}$ where $\lambda_0$ is the starting value of Lagrange multiplier. 

We are now ready to state our main results for Algorithm \ref{algo:algo_name}.

\begin{theorem}
    \label{theorem:general_nonconvex}
    Suppose that Assumptions \ref{assumption:nice_functions} - \ref{assumption:extra_smooth_f} hold, and parameters and step-sizes are chosen such that $\lambda_0 \ge 2l_{f,1} / \mu_g$ and
    \begin{subequations}
    \label{eq:step_size_theorem}
    \begin{align}
        & \beta_k \le \gamma_k \le \min \left(\frac{1}{4 l_{g,1}}, \frac{1}{4 T \mu_g} \right), \ \alpha_k \le \min \left( \frac{1}{8 l_{f,1}}, \frac{1}{2 \xi l_{F,1}} \right), \label{eq:step_size_theorema} \\
        & \frac{\xi}{T} < c_\xi \mu_g \cdot \max \left( l_{g,1} l_{*,0}^2, \ l_{*,1} \sqrt{M} \right)^{-1}, \quad \frac{\delta_k}{\lambda_k} \le \frac{T \mu_g \beta_k}{16} \label{eq:step_size_theoremb}
    \end{align}
    \end{subequations}
    for all $k \ge 0$ with a proper absolute constant $c_\xi > 0$. Then for any $K \ge 1$, Algorithm \ref{algo:algo_name} iterates satisfy
    \begin{align*}
        \sum_{k=0}^{K-1} \xi \alpha_k \Exs[\|\grad F(x_k)\|^2] &\le O_{\texttt{P}} (1) \cdot \sum_{k} \xi\alpha_k\lambda_k^{-2} + O_{\texttt{P}}(\sigma_f^2) \cdot \sum_{k} \alpha_k^2 \lambda_k + O_{\texttt{P}}(\sigma_g^2) \cdot \sum_{k} \gamma_k^2 \lambda_k + O_{\texttt{P}} (1), 
    \end{align*}
    where $O_{\texttt{P}}(1)$ are instance-dependent constants.
\end{theorem}
The proof of Theorem~\ref{theorem:general_nonconvex} is given in Appendix~\ref{appendix:main_proof_non_convex}. At a high level, our analysis investigates the decrease in expectation (with $k$) of the potential function $\mathbb{V}_k$ defined by
\begin{align}
\label{eq:pot0}
    \mathbb{V}_k := &(F(x_k) - F^*) + l_{g,1} \lambda_k \|y_k - y_{\lambda_k}^* (x_k)\|^2  + \frac{\lambda_k l_{g,1}}{2} \|z_k - y^*(x_k)\|^2,
\end{align}
where $F^*$ is the minimum value of $F$ and $y^*_{\lambda}$ and $y^*$ are given in \eqref{eq:yls} and \eqref{problem:bilevel}, respectively. That is, in addition to the decrease in values of $F$ and $z_k - y_k^*$ which have been standardized in literature, we track the error between $y_k$ and $y_{\lambda_k}^* (x_k)$ since $y_{\lambda, k}^*$ is the key to compute true $\grad F(x_k)$ only with gradients. It is also shown in the proof that the right scaling factor for the tracking errors is $O_{\texttt{P}}(\lambda_k)$. 

We now describe how we design step sizes. Note that the conditions \eqref{eq:step_size_theorema} are standard conditions on the step sizes for gradient-based methods with smooth functions. The conditions \eqref{eq:step_size_theoremb} arise from the double-loop nature of the problem, as discussed in Section \ref{sec:step}. In accordance with the step-size design rule \eqref{eq:step_size_theorem}, we propose the following: 
\begin{align}
    &T = \max \left(32, (c_{\xi} \mu_g)^{-1} \max \left(l_{g,1} l_{*,0}^2, \ \sqrt{M} l_{*,1} \right) \right), \nonumber \\
    &\xi = 1, \alpha_k = \frac{c_\alpha}{(k+k_0)^a}, \ \gamma_k = \frac{c_\gamma}{(k+k_0)^c}, \label{eq:step_size_design}
\end{align}
and for the multiplier increase sequence $\{\delta_k\}$,
\begin{align}
    \delta_k = \min\left( \frac{T \mu_g }{16} \alpha_k \lambda_k^2, \ \frac{\gamma_k}{2\alpha_k} - \lambda_k \right), \label{eq:step_size_design_delta}
\end{align}
with some rate constants $a,c \in [0,1]$ and $a \ge c$. We design the starting value $\lambda_0$ of the Lagrange multiplier 
and the constants as
\begin{align}
    &k_0 \ge \frac{4}{\mu_g} \max \left( \frac{ \xi l_{F,1}}{2}, T l_{g,1}, l_{f,1} \right), \ \lambda_0 \ge \frac{2l_{f,1}}{\mu_g} , \nonumber \\
    &c_\gamma = \frac{1}{\mu_g k_0^{1-c}}, \ c_\alpha = \frac{1}{2\lambda_0 \mu_g k_0^{1-a}}. \label{eq:step_size_constant_design}
\end{align}
These choices simplify the convergence rate analysis, but any set of choices can be used as long as it satisfies \eqref{eq:step_size_theorem}. 
With the choices above, we can specify the rate of convergence in three different regimes of stochastic noises.
\begin{corollary}
    \label{corollary:non_convex_F}
    Suppose that the conditions of Theorem~\ref{theorem:general_nonconvex} hold,  with step-sizes designed as in \eqref{eq:step_size_design}, \eqref{eq:step_size_design_delta}, and
    \eqref{eq:step_size_constant_design}. Let $R$ be a random variable drawn from a uniform distribution over $\{0, ..., K-1\}$. 
    Then the following convergence results hold after $K$ iterations of Algorithm \ref{algo:algo_name}.
    \begin{enumerate}
        \item[(a)]  If stochastic noises are present in both upper-level objective $f$ and lower-level objective $g$ ({\it i.e.,} $\sigma_f^2, \sigma_g^2 > 0$), then by setting $a=5/7$ and $c=4/7$ in \eqref{eq:step_size_design} and \eqref{eq:step_size_constant_design}, we obtain 
        $\Exs[\|\grad F(x_R)\|^2] \asymp \frac{\log K}{K^{2/7}}$,
        \item[(b)] If stochastic noises are present only in $f$ ({\it i.e.,} $\sigma_f^2 > 0)$, $\sigma_g^2 = 0$), then by setting $a = 3/5$ and $c=2/5$ in \eqref{eq:step_size_design} and \eqref{eq:step_size_constant_design}, we obtain $\Exs[\|\grad F(x_R)\|^2] \asymp \frac{\log K}{K^{2/5}}$,
        \item[(c)] If we have access to exact information about $f$ and $g$ ({\it i.e.,} $\sigma_f^2 = \sigma_g^2 = 0$), then by setting $a = 1/3$ and $c = 0$ in \eqref{eq:step_size_design} and \eqref{eq:step_size_constant_design}, we obtain $\|\grad F(x_K)\|^2 \asymp \frac{\log K}{K^{2/3}},$
    \end{enumerate}
\end{corollary}
As these results show, stronger convergence results can be proved when noise is present in fewer places in the problem. 
If stochastic noise is present only in the upper-level rather than in both levels, the rate can be improved from $O(k^{-2/7})$ to $O(k^{-2/5})$. In deterministic settings (no noise), we get a rate of  $O(k^{-2/3})$. 
This rate compares to the $O(k^{-1})$ rate that can be obtained with second-order based methods.

\subsection{Main Result for Algorithm \ref{algo:algo_name2}}
When we use the momentum-assisting technique, we require the stochastic functions to be well-behaved as well.
\begin{assumption}
    \label{assumption:nice_stochastic_fg}
    Assumption \ref{assumption:nice_functions} holds for $f(x,y;\zeta)$ and $g(x,y;\phi)$ with probability $1$. 
\end{assumption}
One technical benefit of the momentum technique is that now we no longer require the bounded-gradient assumption w.r.t. $x$ (Assumption~\ref{assumption:bounded_grad_x}) or the smoothness of Hessian of $f$ (Assumption~\ref{assumption:extra_smooth_f}) for the analysis, as we no longer make use of the smoothness of $y_\lambda^*$.
We show the following convergence result for Algorithm \ref{algo:algo_name2}.
\begin{theorem}
    \label{theorem:general_momentum}
    Suppose Assumptions \ref{assumption:nice_functions}-\ref{assumption:gradient_variance} and \ref{assumption:nice_stochastic_fg} hold. If step-size parameters are chosen such that $\lambda_0 \ge 2l_{f,1} / \mu_g$ and 
    \begin{subequations}
        \label{eq:step_size_theorem_momentum}
        \begin{align}
        & \beta_k \le \gamma_k \le \frac{1}{16 l_{g,1}}, \ \xi \alpha_k \le \frac{1}{l_{F,1}},  \ \xi \le c_\xi \frac{\mu_g}{l_{g,1} l_{*,0}^2}, \ \frac{\delta_k}{\lambda_k} \le \frac{\mu_g \beta_k}{8} ,\label{eq:step_size_theorem_momentum_a} \\
            & \eta_0 = \eta_1 = 1, \ \max \left( 2\frac{\gamma_{k-1} - \gamma_k}{\gamma_{k-1}}, c_\eta \frac{l_{g,1}^3}{\mu_g} \gamma_k^2 \right) \le \eta_{k+1} \le 1, \ \delta_k / \gamma_k = o(1), \label{eq:step_size_theorem_momentum_b}
        \end{align}
    \end{subequations}
    with proper absolute constants $c_{\xi}, c_{\eta} > 0$, then for any $K \ge 1$, Algorithm \ref{algo:algo_name2} satisfy
    \begin{align*}
        \sum_{k=0}^{K-1} \xi \alpha_k \Exs[\|\grad F(x_k)\|^2] &\le O_{\texttt{P}} (1) \cdot \sum_{k} \xi\alpha_k\lambda_k^{-2} + O_{\texttt{P}}(\sigma_f^2) \cdot \sum_{k} \frac{\eta_{k+1}^2}{\gamma_k \lambda_k} + O_{\texttt{P}}(\sigma_g^2) \cdot \sum_{k} \frac{\eta_{k+1}^2 \lambda_k}{\gamma_k} + O_{\texttt{P}} (1), 
    \end{align*}
    where $O_{\texttt{P}}(1)$ are instance-dependent constants. 
\end{theorem}
The proof of Theorem \ref{theorem:general_momentum} appears in Appendix \ref{appendix:momentum_method}. 
We introduce the following  step-size design, consistent with \eqref{eq:step_size_theorem_momentum}.
\begin{subequations}
    \label{eq:step_size_momentum_detail}
    \begin{align}
        &\alpha_k = \frac{c_\alpha}{(k+k_0)^a}, \ \gamma_k = \frac{c_\gamma}{(k+k_0)^c}, \ \eta_{k} = (k+1)^{-2c} \\
        &\xi \le c_\xi \frac{\mu_g}{l_{g,1} l_{*,0}^2}, \ \delta_k = \frac{\gamma_k}{\alpha_k} - \lambda_k, \ \lambda_0 \ge \frac{2l_{f,1}}{\mu_g}, \\
        &k_0 \ge \frac{128}{\mu_g} \max\left(\xi l_{F,1}, l_{g,1} \sqrt{\frac{c_\eta l_{g,1}}{\mu_g}} \right), c_\gamma = \frac{8}{\mu_g k_0^{1-c}}, \ c_\alpha = \frac{8}{\mu_g \lambda_0 k_0^{1-a}},
    \end{align}
\end{subequations}
with some rate constants $a,c \in [0,1]$ and $a \ge c$. As a corollary, we can obtain faster convergence rates for Algorithm~\ref{algo:algo_name2} than Algorithm~\ref{algo:algo_name}. 
\begin{corollary}
    \label{corollary:non_convex_momentum}
   Suppose the conditions of Theorem~\ref{theorem:general_momentum} hold.
   Suppose that Algorithm~\ref{algo:algo_name2} is run with step-sizes are designed as in \eqref{eq:step_size_momentum_detail}. Let $R$ be a random variable drawn from a uniform distribution over $\{0, ..., K-1\}$.
   Then the following convergence results hold after $K$ iterations of Algorithm \ref{algo:algo_name2}.
    \begin{enumerate}
        \item[(a)]  If stochastic noises are present in both upper-level objective $f$ and lower-level objective $g$ ({\it i.e.,} $\sigma_f^2, \sigma_g^2 > 0$), then by setting $a=3/5$ and $c=2/5$ in~\eqref{eq:step_size_momentum_detail}, we obtain 
        $\Exs[\|\grad F(x_R) \|^2] \asymp \frac{\log K}{K^{2/5}}$.
        \item[(b)] If stochastic noises are present only in $f$ ({\it i.e.,} $\sigma_f^2 > 0)$, $\sigma_g^2 = 0$), then by setting $a = 1/2$ and $c=1/4$ in~\eqref{eq:step_size_momentum_detail}, we obtain $\Exs[\|\grad F(x_R)\|^2] \asymp \frac{\log K}{K^{1/2}}$.
        \item[(c)] If we have access to exact information about $f$ and $g$ ({\it i.e.,} $\sigma_f^2 = \sigma_g^2 = 0$), then by setting $a = 1/3$ and $c = 0$ in~\eqref{eq:step_size_momentum_detail}, we obtain $\|\grad F(x_K)\|^2 \asymp \frac{\log K}{K^{2/3}}$.
    \end{enumerate}
\end{corollary}
The improvements in rates are different in different stochasticity regimes. 
For instance, the sample complexity required to achieve $\epsilon$-stationary point is $\tilde{O}_{\texttt{P}}(\epsilon^{-7/2})$ without momentum and $\tilde{O}_{\texttt{P}} (\epsilon^{-5/2})$ with momentum --- a factor of $O(\epsilon)$ improvement --- when stochastic noises are present in both levels. 
In contrast, when stochastic noises are only in the upper-level objective, then the overall sample complexity is tightened from $\tilde{O}_{\texttt{P}}(\epsilon^{-5/2})$ to $\tilde{O}_{\texttt{P}}(\epsilon^{-2})$, an $O(\epsilon^{-0.5})$ improvement. 
Whether Algorithm~\ref{algo:algo_name2} achieves the optimal sample complexity for fully first-order methods is an interesting topic for future work.

\subsection{Discussion}
Because our algorithms do not access second-order derivatives of $g$, their iteration convergence rate is slower, decreasing from  $O(k^{-1/2})$ (e.g., \cite{chen2021closing}) to $O(k^{-2/7})$ for algorithms without momentum and  from $O(k^{-2/3})$ (e.g., \cite{khanduri2021near}) to $O(k^{-2/5})$ for algorithms with momentum. 
This is not unexpected since we use less information. 
Our experiments, perhaps surprisingly, do not show a slowdown in the convergence speed.
In fact, first-order methods even outperform existing methods that use second-order information of $g$, as we show in Section \ref{section:experiment}.
We add that in practice, if a bias of $O(1/\lambda^2)$-bias in the solution is not critical to the overall performance, then we can set $\lambda_k := \lambda$ constant at all iterations and choose more aggressive step-sizes, {\it e.g,} $\alpha_k \asymp k^{-1/2}, \gamma_k \asymp k^{-1/2}$ as in \cite{chen2021closing}.
Such a strategy yields faster convergence to a certain biased point. 

When deterministic gradient oracles are available, the authors in \cite{ye2022bome} employed the so called {\it dynamic-barrier} method \cite{gong2021automatic} to decide the value of $\lambda_k$ at every iteration, based on $\|\grad_y g(x_k,z_{k+1}) - \grad_y g(x_k,y_{k+1})\|$. 
Such an approach requires precise knowledge of the latter quantity, which is not available in stochastic settings. 
Our result shows that  a simple design of polynomial-rate growth of $\lambda_k$ is sufficient; an adaptive choice is not needed for good practical performance. 
Further, the convergence rate reported in \cite{ye2022bome} is $k^{-1/4}$, while our result guarantees $k^{-2/3}$ convergence rate in deterministic settings. 

\section{Experiments}
\label{section:experiment}

\begin{figure}[t]
    \centering
    \begin{tabular}{cc}
        \includegraphics[width=75mm]{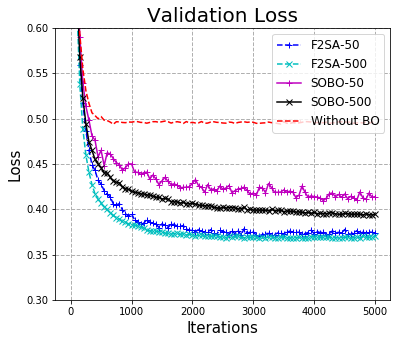} &
        \includegraphics[width=75mm]{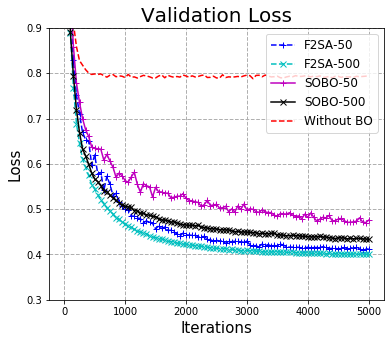} \\
        (a) & (b)
    \end{tabular}
    \caption{Outer objective (validation) loss with label corruption rate: (a) $p=0.1$, (b) $p=0.3$.}
    \label{fig:basic_experiment}
\end{figure}

We demonstrate the proposed algorithms on a data hyper-cleaning task involving MNIST \cite{deng2012mnist}.
We are given a noisy training set $\mD_{\text{train}} := \{(\tilde{x}_i, \tilde{y}_i)\}_{i=1}^{n}$ with the label $\tilde{y}_i$ being randomly corrupted with probability $p < 1$. 
We are also given a small but clean validation set $\mD_{\text{val}} := \{(x_i, y_i)\}_{i=1}^{m}$. 
The goal is to assign weights to each training data point so that the model trained on the weighted training set yields good performance on the validation set. 
This task can be formulated as bilevel opimization problem, as follows:
\begin{align*}
    \min_{\lambda} &\qquad \qquad   \tssum_{i=1}^m l(x_i, y_i; w^*) \\
    \text{s.t.}  &\qquad  w^* \in \arg \min_{w} \tssum_{i=1}^n \sigma(\lambda_i) l(\tilde{x}_i, \tilde{y}_i; w) + c \|w\|^2.  
\end{align*}
where $\sigma(\cdot)$ is a sigmoid function, $l(x,y;w)$ is a logistic loss function with parameter $w$ and $c$ is a regularization constant. We use $n=19000$ training samples and $m=1000$ clean validation samples with regularization parameter $c = 0.01$. 
We do not include momentum-assisted methods in our discussion, since we do not observe a significant improvement over the \algname \,
approach of Algorithm~\ref{algo:algo_name} .

We demonstrate the performance of Algorithm \ref{algo:algo_name} (\algname) 
and the second-order based method (SOBO) with batch sizes $50$ and $500$. 
We note that several existing second-order methods are in principle the same when momentum or variance-reduction techniques are omitted \cite{ghadimi2018approximation, hong2020two, chen2021closing}, so we use the implementation of stocBiO \cite{ji2021bilevel} as a representative of the other second-order methods.
As a baseline, we also add a result from training without bilevel formulation (Without BO), {\it i.e.,} train on all samples as usual, ignoring the label corruption. 
Results are shown in Figure~\ref{fig:basic_experiment}.\footnote{We report our best results obtained with different hyper-parameters for each algorithm. }

Although iteration complexity is worse for first-order methods than SOBO, we observe that \algname~is at least on par with SOBO in this example. 
It can even give superior performance when the batch size is small. 
We conjecture that stochastic noises in Hessian become significantly larger than those in gradients, degrading the performance of SOBO. 
In our experiment, we also observe that Neumann approximation \cite{ghadimi2018approximation} for estimating the Hessian-inverse may induce non-negligible bias in practice.\footnote{We used degree 5 approximation in our experiment. A larger degree approximation did not yield significant improvement.}
In contrast, our fully first-order method \algname~is much less sensitive to small batch sizes and free of bias.

\section{Future work}
We study fully first-order methods for stochastic bilevel optimization and their non-asymptotic performance. We conclude the paper with discussions on several future directions. 

\paragraph{Lower Bound} As we discussed, we observe a gap between the fully first-order method and existing methods that use second-order information of $g$. We conjecture that this gap is fundamental, and it would be an interesting future question to investigate the fundamental limits of fully first-order methods.

\paragraph{More Lower-Level Problems} There are already several recent work that considers a more challenging case when the lower-level optimization problem can be non-strongly-convex \cite{liu2021towards, liu2021value, arbel2022non} and non-smooth \cite{lu2023first}. The potential benefit of the first-order method over existing second-order based methods is that it can still be considered to tackle such scenarios, whereas the formula \eqref{eq:gradF} is only available for well-conditioned lower-level problems. We believe it is an important future direction to study a more general class of \eqref{problem:bilevel} beyond strongly-convex lower-level problems with fully first-order methods. Adding variable-dependent constraints to the lower-level problem would also lead to an interesting extension of fully first-order approaches.

\bibliographystyle{abbrv}
\bibliography{main}

\newpage 

\appendix

\begin{appendices}

\section{Auxiliary Lemmas}
All deferred proofs in the main text and appendix are directed to Appendix \ref{appendix:deferred_proof}.

\subsection{Additional Notation}
\begin{table}[H]
    \centering
    \begin{tabular}{| c | l | l |}
        \hline
         \textrm{Symbol} & \textrm{Meaning} & \textrm{Less than} \\ \hline
         $l_{f,0}$ & Bound of $\|\grad_x f\|, \|\grad_y f\|$ & $\cdot$ \\ \hline
         $l_{f,1}$ & Smoothness of $f$ & $\cdot$ \\ \hline
         $l_{g,0}$ & Bound of $\|\grad_x g\|$ & $\cdot$ \\ \hline
         $l_{g,1}$ & Smoothness of $g$ & $\cdot$ \\ \hline
         $\mu_g$ & Strong-convexity of $g$ & $\cdot$ \\ \hline
         $l_{g,2}$ & Hessian-continuity of $g$ & $\cdot$ \\ \hline
         $M_f$ & Second-order moment of $\grad f(x,y;\zeta)$ & $l_{f,0}^2 + \sigma_f^2$ \\ \hline
         $M_g$ & Second-order moment of $\grad g(x,y;\phi)$ & $l_{g,0}^2 + \sigma_g^2$ \\ \hline
         $l_{f,2}$ & Hessian-continuity of $f$ (with Assumption \ref{assumption:extra_smooth_f}) & $\cdot$ \\ \hline
         $l_{F,1}$ & Smoothness of $F(x)$ & $l_{*,0} \left(l_{f,1} + \frac{l_{g,1}^2}{\mu_g} + \frac{2l_{f,0}l_{g,1}l_{g,2}}{\mu_g^2} \right)$\\ \hline
         $l_{\lambda,0}$ & Lipschitzness of $y_\lambda^*(x)$ (for all $\lambda \ge 2l_{f,1}/\mu_g$) & $\frac{3 l_{g,1}}{\mu_g}$ \\ \hline
         $l_{\lambda,1}$ & Smoothness of $y_\lambda^*(x)$ (for $\lambda \ge 2l_{f,1}/\mu_g$ with Assumption \ref{assumption:extra_smooth_f}) & $32 (l_{g,2} + \lambda^{-1} \cdot l_{f,2}) \frac{l_{g,1}^2}{\mu_g^3}$ \\ \hline
         $l_{*,0}$ & $= 1 + \max_{\lambda \ge 2l_{f,1}/\mu_g}l_{\lambda,0}$ & $\cdot$ \\ \hline
         $l_{*,1}$ & $= \max_{\lambda \ge 2l_{f,1}/\mu_g} l_{\lambda, 1}$ & $\cdot$ \\ \hline
    \end{tabular}
    \caption{Meaning of Constants}
    \label{tab:constant_relations}
\end{table}
To simplify the representation for the movement of variables, we often use $q_k^x$, $q_k^y$ and $q_k^z$ defined as
\begin{align}
    q_k^x &:= \grad_x f(x_k, y_{k+1}) + \lambda_k (\grad_x g(x_k, y_{k+1}) - \grad_x g(x_k, z_{k+1})), \nonumber \\
    q_{k,t}^y &:= \grad_y f(x_k, y_{k,t}) + \lambda_k \grad_y g(x_k, y_{k,t}), \nonumber \\
    q_{k,t}^z &:= \grad_y g(x_k, z_{k,t}). \label{eq:qk_def}
\end{align}
The above quantities are the expected movements of $x_k, y_k^{(t)}, z_k^{(t)}$ respectively if there are no stochastic noises in gradient oracles. We also summarize symbols and their meanings for instance-specific constants in Table \ref{tab:constant_relations}.

\subsection{Auxiliary Lemmas}
We first state a few lemmas that will be useful in our main proofs.

\begin{lemma}
    \label{lemma:outer_F_smooth}
    $F(x) = f(x, y^*(x))$ is $l_{F,1}$-smooth where
    \begin{align*}
        l_{F,1} \le l_{*,0} \left(l_{f,1} + \frac{l_{g,1}^2}{\mu_g} + \frac{2l_{f,0}l_{g,1}l_{g,2}}{\mu_g^2} \right).
    \end{align*}
\end{lemma}

\begin{lemma}
    \label{lemma:relation_Lagrangian_F}
    For any $x,y, \lambda > 0$, the following holds:
    \begin{align*}
        &\|\grad F(x) - \grad_x \mL_\lambda(x,y) + \grad_{xy}^2 g(x,y^*(x))^\top \grad^2_{yy} g(x,y^*(x))^{-1} \grad_y \mL_\lambda(x,y)\| \\
        &\qquad \le 2 (l_{g,1}/\mu_g) \|y-y^*(x)\| \left(l_{f,1} + \lambda \cdot \min(2l_{g,1}, l_{g,2} \|y-y^*(x)\|)\right).
    \end{align*}
\end{lemma}

\begin{lemma}
    \label{lemma:nice_y_star_lagrangian}
    Under Assumptions \ref{assumption:nice_functions}, \ref{assumption:extra_nice_g} and \ref{assumption:extra_smooth_f}, and $\lambda > 2l_{f,1} / \mu_g$, a function $y_{\lambda}^*(x)$ is $l_{\lambda,1}$-smooth: for any $x_1, x_2 \in X$, we have $$\|\grad y_\lambda^*(x_1) - \grad y_\lambda^*(x_2)\| \le l_{\lambda, 1} \|x_1 - x_2\|$$ where $l_{\lambda,1} \le 32 (l_{g,2} + \lambda^{-1} l_{f,2}) l_{g,1}^2 / \mu_g^3$.
\end{lemma}


\begin{lemma}
    \label{lemma:y_star_contraction}
    For any fixed $\lambda > 2l_{f,1} / \mu_g$, at every $k$ iteration conditioned on $\mathcal{F}_k$, we have
    \begin{align*}
        \Exs[\|y^*(x_{k+1}) - y^*(x_k)\|^2| \mathcal{F}_k] \le \xi^2 l_{*,0}^2 \left( \alpha_k^2 \Exs[\|q_k^x\|^2 | \mathcal{F}_k] + \alpha_k^2 \sigma_f^2 + \beta_k^2 \sigma_g^2 \right). 
    \end{align*}
\end{lemma}

\begin{lemma}
    \label{lemma:y_star_smoothness_bound}
    At every $k^{th}$ iteration, conditioned on $\mathcal{F}_k$, let $v_k$ be a random vector decided before updating $x_k$. Then for any $\eta_k > 0$, we have
    \begin{align*}
        \Exs[\vdot{v_k}{y^*(x_{k+1}) - y^*(x_k)} | \mathcal{F}_k ] &\le (\xi \alpha_k \eta_k + M \xi^2 l_{*,1}^2 \beta_k^2) \Exs[\|v_k\|^2 | \mathcal{F}_k] \\
        &\quad + \left(\frac{\xi \alpha_k l_{*,0}^2}{4\eta_k} + \frac{\xi^2 \alpha_k^2}{4}\right) \Exs[\|q_k^x\|^2 | \mathcal{F}_k] + \frac{\xi^2}{4}(\alpha_k^2 \sigma_f^2 + \beta_k^2 \sigma_g^2),
    \end{align*}
    where $M := \max\left(l_{f,0}^2 + \sigma_f^2, l_{g,0}^2 + \sigma_g^2\right)$.
\end{lemma}

\begin{lemma}
    \label{lemma:y_star_lambda_vdot_bound}
    Under Assumptions \ref{assumption:nice_functions}-\ref{assumption:extra_smooth_f}, at every $k^{th}$ iteration, conditioned on $\mathcal{F}_k$, let $v_k$ be a random vector decided before updating $x_k$. Then for any $\eta_k > 0$, we have
    \begin{align*}
        \Exs[\vdot{v_k}{y_{\lambda_{k+1}}^*(x_{k+1}) - y_{\lambda_k}^*(x_k)} | \mathcal{F}_k ] &\le (\delta_k/\lambda_k + \xi \alpha_k \eta_k + M \xi^2 l_{\lambda_k,1}^2 \beta_k^2) \Exs[\|v_k\|^2 | \mathcal{F}_k] \\
        &+ \left(\frac{\xi \alpha_k l_{*,0}^2}{4\eta_k} + \frac{\xi^2 \alpha_k^2}{4}\right) \Exs[\|q_k^x\|^2 | \mathcal{F}_k] + \frac{\xi^2}{4} (\alpha_k^2 \sigma_f^2 + \beta_k^2 \sigma_g^2) + \frac{\delta_k l_{f,0}^2}{\lambda_k^3 \mu_g^2},
    \end{align*}
    where $M := \max\left(l_{f,0}^2 + \sigma_f^2, l_{g,0}^2 + \sigma_g^2\right)$.
\end{lemma}

\section{Main Results for Algorithm \ref{algo:algo_name}}
\label{appendix:main_proof_non_convex}

In this section, we prove our key estimate, Theorem \ref{theorem:general_nonconvex}. Our aim is to find the upper bound of  $\mathbb{V}_{k+1} - \mathbb{V}_k$ for the potential function $\mathbb{V}_k$ given in \eqref{eq:pot0}. For $x_k$ and $y_k$ given in Algorithm~\ref{algo:algo_name}, the following notations will be used:
\begin{align}
\label{eq:ikjk}
    \I_k := \|y_k - y_{\lambda,k}^*\|^2 \hbox{ and } \J_k := \|z_k - y_k^*\|^2
\end{align}
where $y_{\lambda,k}^* := y_{\lambda_k}^*(x_k)$, $y_k^* := y^*(x_k)$, and $x^* = \arg \min_x F(x)$. Recall that $y^*_{\lambda}$ and $y^*$ are given in \eqref{eq:yls} and \eqref{problem:bilevel}, respectively.
Using the above notation, the potential function given in \eqref{eq:pot0} can be rewritten as  \begin{align}
\label{eq:pot}
    \mathbb{V}_k := (F(x_k) - F(x^*)) + \lambda_k l_{g,1} \I_k + \frac{\lambda_k l_{g,1}}{2} \J_k
\end{align}
for each $k \in \mathbb{N}$. In the following three subsections, we find the upper bound of $\mathbb{V}_{k+1} - \mathbb{V}_k$ in terms of $\I_k$ and $\J_k$. The proof of Theorem \ref{theorem:general_nonconvex} is given in Section~\ref{sec:pfgeneral}.

\subsection{Estimation of $F(x_{k+1}) - F(x_k)$}
\label{sec:estf}

The step size $\alpha_k$ is designed to satisfy
\begin{align}
    \textbf{(step-size rule):} \qquad \alpha_k \le \frac{1}{2\xi l_{F,1}}, \label{eq:step_cond_alpha_standard_F}
\end{align}
which is essential to obtain the negative term $- \frac{\xi \alpha_k}{4} \|q_k^x\|^2$ on the right hand side of \eqref{eq:fxfx}. This negativity plays an important role in the proof of Theorem \ref{theorem:general_nonconvex} in Section~\ref{sec:pfgeneral}.

On the other hand, we also impose
\begin{align}
    \textbf{(step-size rule):} \qquad \frac{\xi}{T} \le \frac{\mu_g}{96 l_{g,1}}
    \label{eq:step_cond_xi_beta}.
\end{align} The terms, $\|y_{k+1} - y_{\lambda, k}^*\|^2$ and $\|z_{k+1} - y_k^*\|^2$, in the upper bound \eqref{eq:fxfx} will be estimated in Lemma~\ref{lem:gen20} and Lemma~\ref{lem:z_intermediate_contraction}, respectively. 

\begin{proposition}
\label{prop:g1}
Under the step-size rules given in \eqref{eq:step_cond_alpha_standard_F}, and \eqref{eq:step_cond_xi_beta} and $\lambda_k \ge 2 l_{f,1} / \mu_g$, it holds that for each $k \in \mathbb{N}$
\begin{align}
\label{eq:fxfx}
    \Exs[ F(x_{k+1}) - F(x_k) |\mathcal{F}_k] &\le -\frac{\xi \alpha_k}{4} \left( 2 \|\grad F(x_k)\|^2 + \|q_k^x\|^2 \right)  + \frac{T \mu_g \alpha_k \lambda_k^2}{32} \left(
    2\|y_{k+1} - y_{\lambda, k}^*\|^2 + \|z_{k+1} - y_k^*\|^2 \right) \nonumber \\
    &\quad + \frac{\xi^2 l_{F,1}}{2} (\alpha_k^2 \sigma_f^2 + \beta_k^2 \sigma_g^2) + \frac{\xi \alpha_k}{2} \cdot 3 C_\lambda^2 \lambda_k^{-2},
\end{align}
where $q_k^x$ is given in \eqref{eq:qk_def}, and $C_\lambda := \frac{4 l_{f,0} l_{g,1}}{\mu_g^2} \left(l_{f,1} + \frac{2l_{f,0} l_{g,2}}{\mu_g}\right)$.
\end{proposition}

\begin{proof}
From the smoothness of $F$, 
\begin{align*}
    \Exs[ F(x_{k+1}) - F(x_k) |\mathcal{F}_k] &\le \Exs[\vdot{\grad F(x_k)}{x_{k+1} - x_k} + \frac{l_{F,1}}{2} \|x_{k+1} - x_k\|^2|\mathcal{F}_k].
\end{align*}
As $q_k^x$ satisfies $\Exs[x_{k+1} - x_k|\mathcal{F}_k] = \alpha_k q_k^x$, 
\begin{align*}
    &\Exs[ F(x_{k+1}) - F(x_k) |\mathcal{F}_k] = -\xi \alpha_k \vdot{\grad_x F(x_k)}{q_k^x} + \frac{l_{F,1}}{2} \Exs[\|x_{k+1} - x_k\|^2  |\mathcal{F}_k] \\
    &= -\frac{\xi \alpha_k}{2} (\|\grad F(x_k)\|^2 + \|q_k^x\|^2 - \| \grad F(x_k) - q_k^x \|^2) + \frac{l_{F,1}}{2} \Exs[ \|x_{k+1} - x_k\|^2 |\mathcal{F}_k].
\end{align*}
Note that
\begin{align*}
    \Exs[\|x_{k+1} - x_k\|^2] &\le \xi^2 \alpha_k^2 \Exs[\|q_k^x\|^2 + \xi^2 (\alpha_k^2 \sigma_f^2 + \beta_k^2 \sigma_g^2),
\end{align*}
and thus with \eqref{eq:step_cond_alpha_standard_F} we have
\begin{align*}
    &\Exs[ F(x_{k+1}) - F(x_k) |\mathcal{F}_k] \le -\frac{\xi \alpha_k}{2} \|\grad F(x_k)\|^2 - \frac{\xi \alpha_k}{4} \|q_k^x\|^2 \\
    &\quad + \frac{\xi \alpha_k}{2} \|\grad F(x_k) - q_k^x\|^2 + \frac{\xi^2 l_{F,1}}{2} (\alpha_k^2 \sigma_f^2 + \beta_k^2 \sigma_g^2).
\end{align*}

Next, we bound $\|\grad F(x_k) - q_k^x\|$ using the triangle inequality:
\begin{align*}
    \|q_k^x - \grad F(x_k)\| &= \|q_k^x - \grad \mL_{\lambda_k}^*(x_k) + \grad \mL_{\lambda_k}^*(x_k) - \grad F(x_k)\|  \\
    &\le \|\grad_x f(x_k, y_{k+1}) - \grad_x f(x_k, y_{\lambda,k}^*)\| + \lambda_k \|\grad_x g(x_k, y_{k+1}) - \grad_x g(x_k, y_{\lambda, k}^*)\| \\
    &\quad + \lambda_k \|\grad_x g(x_k, z_{k+1}) - \grad_x g(x_k, y_k^*)\| + \| \grad \mL_{\lambda_k}^*(x_k) - \grad F(x_k)\|.
\end{align*}
From Lemma~\ref{lemma:l_star_lambda_approximate}, the term $\| \grad \mL_{\lambda_k}^*(x_k) - \grad F(x_k)\|$ is bounded by $C_\lambda / \lambda_k$. Combining with the regularity of $f$ and $g$ yields the following:
\begin{align}
\label{eq:fxfx1}
    \|q_k^x - \grad F(x_k)\| \le 2 l_{g,1} \lambda_k \|y_{k+1} - y_{\lambda,k}^*\| + l_{g,1} \lambda_k \|z_{k+1} - y_{k}^*\| + C_\lambda / \lambda_k.
\end{align}
Note that $\lambda_k \ge 2 l_{f,1} / \mu_g$, and thus $l_{f,1} < l_{g,1} \lambda_k$.

Finally, from Cauchy-Schwartz inequality $(a+b+c)^2 \le 3 (a^2 + b^2 + c^2)$, we get
\begin{align}
    &\Exs[ F(x_{k+1}) - F(x_k) |\mathcal{F}_k] \le -\frac{\xi \alpha_k}{2} \|\grad F(x_k)\|^2 - \frac{\xi \alpha_k}{4} \|q_k^x\|^2 \\
    &\quad + \frac{\xi \alpha_k}{2} \cdot 3 C_\lambda^2 \lambda_k^{-2} + 3 \xi \alpha_k l_{g,1}  \lambda_k^2 \|z_{k+1} - y_k^*\|^2 + 6 \xi  \alpha_k l_{g,1} \lambda_k^2 \|y_{k+1} - y_{\lambda, k}^*\|^2  + \frac{\xi^2 l_{F,1}}{2} (\alpha_k^2 \sigma_f^2 + \beta_k^2 \sigma_g^2). \nonumber
\end{align}
The step-size condition \eqref{eq:step_cond_xi_beta} concludes our claim.
\end{proof}



\subsection{Descent Lemma for $y_k$ towards $y_{\lambda,k}^*$}

In this section, the upper bounds of $\I_{k+1}$ and $\|y_{k+1} - y_{\lambda, k}^*\|$ are provided, respectively, in Lemma~\ref{lem:gen2} and Lemma~\ref{lem:gen20}. The following rule is required to ensure that $\|y_{k+1} - y_{\lambda,k+1}\|^2$ contracts:
\begin{align}
    \textbf{(step-size rule):} \qquad & \frac{\delta_k}{\lambda_k} \le \frac{T \beta_k \mu_g}{32}, \hbox{ and } 2\xi^2 Ml_{*,1}^2 \beta_k^2 < T \beta_k \mu_g / 16. \label{eq:step_cond_beta_y_lambda}
\end{align}
The first condition holds directly from \eqref{eq:step_size_theoremb}, and the second condition holds since $\beta_k \le \frac{1}{4T\mu_g}$ and also
\begin{align*}
    \frac{\xi^2}{T^2} \le \frac{\mu_g^2}{8} (M l_{*,l}^2)^{-1},
\end{align*}
which also holds by \eqref{eq:step_size_theoremb} with sufficiently small $c_\xi$.



\begin{lemma}
\label{lem:gen2}
    Under the step-size rule  \eqref{eq:step_cond_beta_y_lambda}, it holds that for each $k \in \mathbb{N}$
    \begin{align}
    \Exs[\I_{k+1}| \mathcal{F}_k] &\le \left(1 + T \beta_k \mu_g / 4 \right) \Exs[\|y_{k+1} - y_{\lambda, k}^*\|^2 | \mathcal{F}_k] \nonumber\\
    &\quad + O\left(\frac{\xi^2 l_{*,0}^2 \alpha_k^2 }{\mu_gT \beta_k}\right) \Exs[\|q_k^x\|^2 | \mathcal{F}_k] + O\left( \frac{\delta_k}{\lambda_k^3} \frac{l_{f,0}^2}{\mu_g^2} \right) + O(\xi^2 l_{*,0}^2) \cdot (\alpha_k^2 \sigma_f^2 + \beta_k^2 \sigma_g^2).  \label{eq:y_star_lambda_contraction}
\end{align}
where $\I_k$ and $q_k^x$ are given in \eqref{eq:ikjk} and \eqref{eq:qk_def}, respectively.
\end{lemma}

\begin{proof}
We can start from 
\begin{align*}
    \|y_{k+1} - y_{\lambda, k+1}^*\|^2 = \underbrace{\|y_{k+1} - y_{\lambda, k}^*\|^2}_{(i)} + \underbrace{\|y_{\lambda, k+1}^* - y_{\lambda, k}^*\|^2}_{(ii)} - \underbrace{2\vdot{y_{k+1} - y_{\lambda, k}^*}{y_{\lambda, k+1}^* - y_{\lambda, k}^*}}_{(iii)}. 
\end{align*}
The upper bound of $(i)$ is given in Lemma~\ref{lem:gen20} below. To bound $(ii)$, we invoke Lemma \ref{lemma:y_star_lagrangian_continuity} to get 
\begin{align*}
    (ii): \Exs[\|y_{\lambda, k+1}^* - y_{\lambda, k}^*\|^2|\mathcal{F}_k] &\le \frac{4 \delta_k^2}{\lambda_k^2 \lambda_{k+1}^2} \frac{l_{f,0}^2}{\mu_g^2} + l_{*,0}^2 \Exs[\|x_{k+1} - x_k\|^2 | \mathcal{F}_k] \\
    &\le \frac{4 \delta_k^2}{\lambda_k^4} \frac{l_{f,0}^2}{\mu_g^2} + \xi^2 l_{*,0}^2(\alpha_k^2 \Exs[\|q_k^x\|^2] + \alpha_k^2 \sigma_f^2 + \beta_k^2 \sigma_f^2).
\end{align*}

For $(iii)$, recall the smoothness of $y_{\lambda}^* (x)$ in Lemma \ref{lemma:nice_y_star_lagrangian}, and thus Lemma \ref{lemma:y_star_lambda_vdot_bound}. By setting $v = y_{k+1} - y_{\lambda, k}^*$ and $\eta_k = T \mu_g \lambda_k / (16 \xi)$, and get
\begin{align*}
    (iii)&\le (2\delta_k / \lambda_k + T \beta_k \mu_g / 8 + 2 M \xi^2 l_{*,1}^2 \beta_k^2) \Exs[\|y_{k+1}-y_{\lambda,k}^*\|^2 |\mathcal{F}_k] \nonumber \\
    &\quad + \xi^2 \left(\frac{\alpha_k^2}{2} + \frac{8 \alpha_k^2 l_{*,0}^2}{\mu_g T \beta_k} \right) \|q_k^x\|^2 + \frac{\xi^2}{2} (\alpha_k^2 \sigma_f^2 + \beta_k^2 \sigma_g^2) + \frac{2 \delta_k}{\lambda_k^3} \frac{l_{f,0}^2}{\mu_g^3}.    
\end{align*}

We sum up the $(i), (ii), (iii)$ to conclude 
\begin{align}
    \Exs[\I_{k+1}| \mathcal{F}_k] &\le \left(1 + 2\delta_k/\lambda_k +  T \beta_k \mu_g / 8 + 2M\xi^2 l_{*,1}^2 \beta_k^2 \right) \Exs[\|y_{k+1} - y_{\lambda, k}^*\|^2] \nonumber\\
    &\quad + O\left(\frac{\xi^2 l_{*,0}^2 \alpha_k^2 }{\mu_gT \beta_k}\right) \|q_k^x\|^2  + O\left( \frac{\delta_k}{\lambda_k^3} \frac{l_{f,0}^2}{\mu_g^2} \right) + O(\xi^2 l_{*,0}^2) \cdot (\alpha_k^2 \sigma_f^2 + \beta_k^2 \sigma_g^2).  
\end{align}
Lastly, the step-size rule   \eqref{eq:step_cond_beta_y_lambda} yields our conclusion.
\end{proof}

Next, we note that $\alpha_k$ and $\beta_k$ are chosen to satisfy  
\begin{align}
    \textbf{(step size rules):} \qquad \alpha_k \le \frac{1}{8 l_{f,1}} \hbox{ and } \beta_k \le \frac{1}{8 l_{g,1}}, \label{eq:step_cond_alpha_beta_standard}
\end{align}
Note that $\beta_k \le \frac{1}{8l_{g,1}}$ is given from the step-size condition \eqref{eq:step_size_theorema}, and $\alpha_k \le \frac{1}{8 l_{g,1} \lambda_k} \le \frac{1}{8 l_{f,1}}$ since $\lambda_k \ge l_{f,1} / \mu_g$.

\begin{lemma}
\label{lem:gen20}
Under the step-size rule given in \eqref{eq:step_cond_alpha_beta_standard}, it holds that for each $k \in \mathbb{N}$
\begin{align}
    \Exs[\|y_{k+1} - y_{\lambda,k}^*\|^2 | \mathcal{F}_k] &\le (1 - 3 T \mu_g \beta_k / 4) \I_k + T (\alpha_k^2 \sigma_f^2 + \beta_k^2 \sigma_g^2). \label{eq:y_star_lambda_intermediate_contraction}
\end{align}
\end{lemma}

\begin{proof}
Since $\Exs[y_{k}^{(t+1)} - y_k^{(t)} | \mathcal{F}_k] = - \alpha_k \grad_y q_k^{(t)} = - \alpha_k \grad_y \mL_{\lambda_k} (x_k, y_k^{(t)})$, we have
\begin{align*}
    \Exs[\|y_{k}^{(t+1)} - y_{\lambda, k}^*\|^2| \mathcal{F}_k] = \|y_{k}^{(t)} - y_{\lambda, k}^*\|^2 - 2\alpha_k \vdot{\grad_y q_k^{(t)}}{y_{k}^{(t)} - y_{\lambda, k}^*} + \Exs[\|y_{k}^{(t+1)} - y_k^{(t)}\|^2| \mathcal{F}_k]. 
\end{align*}
As we start from $\lambda_0 \ge 2\mu_f / \mu_g$, all $\mL_k$ is $(\lambda_k\mu_g/2)$-strongly convex in $y$, and we have
\begin{align*}
    \max \left( \frac{\lambda_k \mu_g }{2} \|y_k^{(t)} - y_{\lambda,k}^*\|^2 , \frac{1}{l_{f,1} + \lambda_k l_{g,1}} \|\grad_y q_k^{(t)}\|^2 \right) \le \vdot{\grad_y q_k^{(t)}}{y_k^{(t)} - y_{\lambda, k}^*}. 
\end{align*}
Using $\Exs[\|y_{k}^{(t+1)} - y_k^{(t)}\|^2|\mathcal{F}_k] \le \alpha_k^2 \|\grad_y q_k^{(t)} \|^2 + \alpha_k^2 \sigma_f^2 + \beta_k^2 \sigma_g^2$, have
\begin{align*}
    (i): \Exs[\|y_{k}^{(t+1)} - y_{\lambda, k}^*\|^2|\mathcal{F}_k] &\le (1 - 3 \mu_g \beta_k / 4) \|y_k^{(t)} - y_{\lambda, k}^*\|^2 + (\alpha_k^2 \sigma_f^2 + \beta_k^2 \sigma_g^2), 
\end{align*}
where we use $\alpha_k (l_{f,1} + \lambda_k l_{g,1}) = \alpha_k l_{f,1} + \beta_k l_{g,1} \le 1/4$ if we have \eqref{eq:step_cond_alpha_beta_standard}.
Repeating this $T$ times, we get \eqref{eq:y_star_lambda_intermediate_contraction}. Note that $y_{k+1} = y_{k}^{(T)}$ and $y_{k} = y_{k}^{(0)}$.
\end{proof}

\subsection{Descent Lemma for $z_k$ towards $y^*_{k}$}

Similar to the previous section, we provide the upper bound of $\J_{k+1}$ first and then estimate $\|z_{k+1}-y_{k}^*\|$ that appears in the upper bound. We work with the following step-size condition:
\begin{align}
    \textbf{(step-size rule):} \qquad 2 Ml_{*,1}^2 \xi^2 \beta_k^2 \le T \mu_g \gamma_k / 16, \label{eq:step_cond_beta_square_gamma}
\end{align}
This condition holds since $\beta_k \le \gamma_k$, and $\beta_k \le \frac{1}{4T\mu_g}$ and $\frac{\xi^2}{T^2} \le \frac{\mu_g^2}{8} (M l_{*,1}^2)^{-1}$.
    
\begin{lemma}
    \label{lemma:z_contraction} 
    Under the step-size rule \eqref{eq:step_cond_beta_square_gamma},    
    at each $k^{th}$ iteration, the following holds:
    \begin{align}
    \Exs[\J_{k+1}|\mathcal{F}_k] &\le \left(1 + \frac{3T\gamma_k\mu_g}{8} \right) \cdot \Exs[\|z_{k+1}-y_{k}^*\|^2|\mathcal{F}_k] \nonumber \\
    &\quad + O\left( \frac{\xi^2 \alpha_k^2  l_{*,0}^2}{T \mu_g \gamma_k} \right) \|q_k^x\|^2+  O\left(\xi^2 l_{*,0}^2 \right) (\alpha_k^2 \sigma_f^2 + \beta_k^2 \sigma_g^2) \label{eq:z_contraction}.
\end{align}
\end{lemma}

\begin{proof}
We estimate each term in the following simple decomposition.
\begin{align*}
    \|z_{k+1} - y_{k+1}^*\|^2 = \underbrace{\|z_{k+1} - y_k^*\|^2}_{(i)} + \underbrace{\|y_{k+1}^* - y_k^*\|^2}_{(ii)} - 2 \underbrace{\vdot{z_{k+1}-y_{k}^*}{y_{k+1}^*-y_k^*}}_{(iii)}. 
\end{align*}
Lemma~\ref{lemma:y_star_lagrangian_continuity} implies that
\begin{align*}
    (ii): \Exs[\|y_{k+1}^* - y_k^*\|^2| \mathcal{F}_k] &\le l_{*,0}^2 \xi^2 (\alpha_k^2 \|\nabla_x q_k\|^2 + \alpha_k^2 \sigma_f^2 + \beta_k^2 \sigma_g^2).
\end{align*}
For $(iii)$, we recall Lemma \ref{lemma:y_star_smoothness_bound} with $v_k = z_{k+1} - y_k^*$ and $\eta_k = T \mu_g \gamma_k / (8 \xi \alpha_k)$, we have
\begin{align*}
    (iii): \vdot{z_{k+1} - y_k^*}{y_{k+1}^* - y_k^*} &\le ( T \gamma_k \mu_g / 8 + M \xi^2 l_{*,1}^2 \beta_k^2) \Exs[\|z_{k+1}-y_k^*\|^2 | \mathcal{F}_k] \\
    &\quad + \left( \frac{\xi^2 \alpha_k^2}{4} + \frac{2 \xi^2 \alpha_k^2 l_{*,0}^2}{T \mu_g \gamma_k} \right) \|q_k^x\|^2 + \frac{\xi^2}{4} (\alpha_k^2 \sigma_f^2 + \beta_k^2 \sigma_g^2).
\end{align*}
The above bounds and Lemma~\ref{lem:z_intermediate_contraction}
imply that
\begin{align}
    \Exs[\J_{k+1}|\mathcal{F}_k] &\le \left(1 + \frac{T\gamma_k\mu_g}{4} + 2 M \xi^2 l_{*,1}^2 \beta_k^2\right) \cdot \Exs[\|z_{k+1}-y_{k}^*\|^2|\mathcal{F}_k] \nonumber \\
    &\quad + \xi^2 \alpha_k^2 \cdot \left( l_{*,0}^2 + \frac{4 l_{*,0}^2}{T \mu_g \gamma_k} + \frac{1}{2} \right) \|q_k^x\|^2+ \xi^2 \cdot \left(\frac{1}{2} + l_{*,0}^2 \right) (\alpha_k^2 \sigma_f^2 + \beta_k^2 \sigma_g^2).
\end{align}
Using \eqref{eq:step_cond_beta_square_gamma}, we conclude.
\end{proof}

Next, $\gamma_k$ is chosen to satisfy the following step-size rules:
\begin{align}
    \textbf{(step-size rule):} \qquad l_{g,1} \gamma_k \le 1/4, \qquad T \mu_g \gamma_k \le 1/4, \label{eq:step_cond_gamma_standard}
\end{align}
which directly comes from \eqref{eq:step_size_theorema}. 

\begin{lemma}
\label{lem:z_intermediate_contraction}
If \eqref{eq:step_cond_gamma_standard} holds, then for each $k \in \mathbb{N}$, the following holds:
\begin{align}
    \Exs[\|z_{k+1} - y_k^*\|^2 | \mathcal{F}_k] &\le (1 - 3 T \mu_g \gamma_k / 4) \J_k + T \gamma_k^2 \sigma_g^2. \label{eq:z_intermediate_contraction}
    \end{align}
\end{lemma}

\begin{proof}
We analyze one step iteration of the inner loop: for each $t = 0, \cdots, T-1$,
\begin{align*}
    \|z_{k}^{(t+1)} - y_k^*\|^2 &= \|z_{k}^{(t)} - y_k^*\|^2 + \|z_{k}^{(t+1)} - z_k^{(t)}\|^2 +  2\vdot{z_{k}^{(t+1)}-z_{k}^{(t)}}{z_{k}^{(t)}-y_k^*} \\
    &= \|z_{k}^{(t)} - y_k^*\|^2 + \gamma_k^2 \|h_{gz}^{k,t} \|^2 - 2\gamma_k \vdot{h_{gz}^{k,t}}{z_{k}-y_k^*}.
\end{align*}
Here, $z_{k+1} = z_{k}^{(T)}$ and $z_{k} = z_{k}^{(0)}$. Note that $\Exs[h_{gz}^{k,t}] = \grad_y g(x_k, z_k^{(t)}) = \grad_y g_k(z_k^{(t)})$ where $g_k(z_k^{(t)}):= g(x_k, z_k^{(t)})$. Taking expectation,
\begin{align*}
    \Exs[\|z_{k}^{(t+1)} - y_k^*\|^2 | \mathcal{F}_{k} ] &\le \|z_{k}^{(t)} - y_k^*\|^2 + \gamma_k^2 \|\grad g_k (z_k^{(t)})\|^2 +  \gamma_k^2 \sigma_g^2 - 2\gamma_k \vdot{\grad g_k(z_k^{(t)})}{z_{k}^{(t)} -y_k^*}.
\end{align*}
The strong convexity and smoothness of $g_k$ imply the coercivity and co-coercivity \cite{nesterov2018lectures}, that is,
\begin{align*}
    \max \left( \mu_g \|z_k^{(t)} -y_k^*\|^2, \frac{1}{l_{g,1}} \|\grad g_k (z_k^{(t)}) - \grad g_k(y_k^*)\|^2 \right) \le \vdot{\grad g_k(z_k^{(t)}) - \grad g_k(y_k^*)}{z_{k}^{(t)} - y_k^*}.
\end{align*}
Note that $y_k^*$ minimizes $g_k(y)$. Use this to cancel out $\gamma_k^2 \|\grad g_k (z_k^{(t)})\|^2$, yielding
\begin{align*}
    \Exs[\|z_{k}^{(t+1)} - y_k^*\|^2 | \mathcal{F}_k] &\le \|z_{k}^{(t)} - y_k^*\|^2 + \gamma_k^2 \sigma_g^2 - \gamma_k(1- l_{g,1} \gamma_k) \vdot{\grad g_k(z_k^{(t)})}{z_{k}^{(t)} -y_k^*} \nonumber \\
    &\le (1 - 3 \mu_g \gamma_k / 4) \|z_{k}^{(t)} - y_k^*\|^2 + \gamma_k^2 \sigma_g^2.
\end{align*}
For this to hold, we need a step-size condition \eqref{eq:step_cond_gamma_standard}.
We can repeat this relation for $T$ times, and we get \eqref{eq:z_intermediate_contraction}.
\end{proof}

\subsection{Proof of Theorem \ref{theorem:general_nonconvex}}
\label{sec:pfgeneral}

Recall $\mathbb{V}_k$ given in \eqref{eq:pot0}. In what follows, we examine    
\begin{align*}
    \mathbb{V}_{k+1} - \mathbb{V}_k &= F(x_{k+1}) - F(x_k) + \lambda_{k+1} l_{g,1} \I_{k+1} - \lambda_k l_{g,1} \I_k \\ &+ \frac{\lambda_{k+1} l_{g,1}}{2} \J_{k+1} - \frac{\lambda_k l_{g,1}}{2} \J_k.
\end{align*}
Using the estimate of $F(x_{k+1}) - F(x_k)$ given in Proposition~\ref{prop:g1} and rearranging the terms, we have
\begin{align*}
    \Exs[\mathbb{V}_{k+1} - \mathbb{V}_k | \mathcal{F}_k] 
    &\le - \frac{\xi \alpha_k}{2} \|\grad F(x_k)\|^2  - \frac{\xi \alpha_k}{4} \Exs[\|q_k^x\|^2  |\mathcal{F}_k] + \frac{\xi \alpha_k}{2} \cdot 3 C_\lambda^2 \lambda_k^{-2} +\frac{\xi^2 l_{F,1}}{2} (\alpha_k^2 \sigma_f^2 + \beta_k^2 \sigma_g^2) \\
    &\quad + l_{g,1} \underbrace{\Exs[\lambda_{k+1} \I_{k+1}  + \frac{\lambda_k T \beta_k\mu_g }{16} \|y_{k+1} - y_{\lambda, k}^*\|^2  - \lambda_k \I_k |\mathcal{F}_k]}_{(i)}\\
    &\quad + \frac{l_{g,1}}{2} \underbrace{\Exs[\lambda_{k+1} \J_{k+1} + \frac{\lambda_k T \gamma_k \mu_g }{32} \|z_{k+1} - y_k^*\|^2 - \lambda_k \J_k |\mathcal{F}_k]}_{(ii)}
\end{align*}


\textbf{Estimation of $(i)$: } 
From Lemma~\ref{lem:gen2}, and $\lambda_{k+1} = \lambda_k + \delta_k$ yield that
\begin{align*}
    (i) &\le \lambda_k \left(1 + \frac{5 T \beta_k\mu_g }{16} + \frac{\delta_k}{\lambda_k} \right) \Exs[\|y_{k+1} - y_{\lambda, k}^*\|^2 | \mathcal{F}_k] -\lambda_k \I_k \\
    &\quad + \underbrace{O(\xi^2 l_{\lambda,0}^2) \frac{\lambda_k \alpha_k^2}{\mu_g T \beta_k} \|q_k^x\|^2 + O(\xi^2 l_{*,0}^2) \lambda_k (\alpha_k^2 \sigma_f^2 + \beta_k^2 \sigma_g^2) + O\left( \frac{l_{f,0}^2}{\mu_g^3} \right) \cdot \frac{\delta_k}{\lambda_k^2}}_{(iii)}.  
\end{align*}
Given the step-size rules \eqref{eq:step_cond_beta_y_lambda}, we obtain
\begin{align*}
    (i) &\le \lambda_k \left(1 + \frac{T \beta_k \mu_g}{2}\right) \Exs[\|y_{k+1} - y_{\lambda, k}^*\|^2 | \mathcal{F}_k] -\lambda_k \I_k  + (iii). 
\end{align*}
The estimation of $\|y_{k+1} - y_{\lambda, k}^*\|^2$ from Lemma~\ref{lem:gen20} yields that
\begin{align*}
    (i) &\le - \frac{\lambda_k T \mu_g \beta_k}{4} \I_k + O(\xi^2 l_{*,0}^2) \frac{\alpha_k}{\mu_g T} +  (iii),\\ 
    &= - \frac{\lambda_k T \mu_g \beta_k}{4} \I_k + O(\xi^2 l_{*,0}^2) \frac{\alpha_k}{\mu_g T} + O(T + \xi^2 l_{*,0}^2) \lambda_k (\alpha_k^2 \sigma_f^2 + \beta_k^2 \sigma_g^2) + O\left( \frac{l_{f,0}^2}{\mu_g^3} \right) \cdot \frac{\delta_k}{\lambda_k^2}.
\end{align*}
Here, we use $(1+a/2)(1-3a/4) \leq 1-a/4$ for $a>0$.

\medskip

\textbf{Estimation of $(ii)$: }
Lemma~\ref{lemma:z_contraction} yields that 
\begin{align*}
    (ii) &\le \lambda_k \left(1 + \frac{\delta_k}{\lambda_k} + \frac{3T\gamma_k\mu_g}{8} + \frac{\lambda_k T \beta_k \mu_g }{32}  \right) \Exs[\|z_{k+1} - y_k^*\|^2 | \mathcal{F}_k ] - \lambda_k \J_k \\
    &\quad + \underbrace{O(\xi^2 l_{*,0}^2) \frac{\lambda_{k+1} \alpha_k^2}{T \mu_g  \gamma_k} \|q_k^x\|^2 + O(\xi^2  \lambda_{k+1} l_{*,0}^2) (\alpha_k^2 \sigma_f^2 + \beta_k^2 \sigma_g^2)}_{(iv)}.
\end{align*}
With $\beta_k \le \gamma_k$, and thus $\delta_k / \lambda_k < T\mu_g \gamma_k / 32$, we have that 
\begin{align*}
    (ii) &\le \lambda_k \left(1 + \frac{T\gamma_k\mu_g}{2} \right) \Exs[\|z_{k+1} - y_k^*\|^2 | \mathcal{F}_k ] - \lambda_k \J_k + (iv)
\end{align*}
Similar to the argument for $(i)$ above, Lemma~\ref{lem:z_intermediate_contraction} yields
\begin{align*}
    (ii) &\le - \frac{\lambda_k T \mu_g \gamma_k}{4} \J_k + O(\xi^2 l_{*,0}^2) \frac{\alpha_k \beta_k}{T \mu_g \gamma_k} \|q_k^x\|^2 + O(\xi^2 \lambda_k l_{*,0}^2)  (\alpha_k^2 \sigma_f^2 + \beta_k^2 \sigma_g^2) + O(\lambda_k) T\gamma_k^2 \sigma_g^2.
\end{align*}


\medskip

Plug the bound for $(i)$ and $(ii)$, after rearranging terms, we get
\begin{align*}
    \Exs[\mathbb{V}_{k+1} - \mathbb{V}_k | \mathcal{F}_k] &\le - \frac{\xi \alpha_k}{2} \|\grad F(x_k)\|^2  + \frac{\xi \alpha_k}{2} \cdot 3 C_\lambda^2 \lambda_k^{-2} +\frac{\xi^2 l_{F,1}}{2} (\alpha_k^2 \sigma_f^2 + \beta_k^2 \sigma_g^2) \\
    &\quad - \frac{\xi \alpha_k}{4} \left( 1 - O \left(\frac{\xi l_{g,1} l_{*,0}^2 \beta_k}{\mu_g T \gamma_k} \right) - O\left( \frac{\xi l_{g,1} l_{*,0}^2 }{\mu_g T} \right) \right) \Exs[\|q_k^x\|^2  |\mathcal{F}_k] \\
    &\quad - \frac{\lambda_k l_{g,1} T \mu_g \beta_k}{4} \I_k -\frac{\lambda_k l_{g,1} T 
    \mu_g \gamma_k}{4} \J_k \\
    &\quad +  O(T + \xi^2 l_{*,0}^2) \cdot l_{g,1} \lambda_k (\alpha_k^2 \sigma_f^2 + (\beta_k^2 + \gamma_k^2) \sigma_g^2) + O\left(\frac{l_{g,1} l_{f,0}^2}{\mu_g^3} \right) \frac{\delta_k}{\lambda_k^2},
\end{align*}
A crucial step here is to ensure that terms driven by $\Exs[\|q_k^x\|^2]$ is negative. 
To ensure this, we require 
\begin{align*}
    \textbf{(step-size rules):} \qquad &\xi l_{g,1} l_{*,0}^2 \beta_k \le c_1 \mu_g T \gamma_k, \nonumber \\ 
    & \xi l_{g,1} l_{*,0}^2 \le c_2 \mu_g T , 
\end{align*}
for some absolute constants $c_1, c_2 > 0$, which holds by $\beta_k \le \gamma_k$ and \eqref{eq:step_size_theoremb} with sufficiently small $c_\xi > 0$. Once this holds, we can conclude that 
\begin{align*}
    \Exs[\mathbb{V}_{k+1} - \mathbb{V}_k | \mathcal{F}_k] &\le - \frac{\xi \alpha_k}{2} \|\grad F(x_k)\|^2 - \frac{\lambda_k T \mu_g \gamma_k}{4} \|z_k - y_k^*\|^2 - \frac{\lambda_k T \mu_g \beta_k}{4} \|y_k - y_{\lambda,k}^*\|^2 \nonumber \\
    &\quad + O(\xi C_{\lambda}^2) \frac{\alpha_k}{\lambda_k^2} + O\left(\frac{l_{g,1} l_{f,0}^2}{\mu_g^3} \right) \frac{\delta_k}{\lambda_k^2} + O(\xi^2 l_{F,1}) (\alpha_k^2 \sigma_f^2 + \beta_k^2 \sigma_g^2) \nonumber \\
    &\quad + O(T + \xi^2 l_{*,0}^2) \cdot l_{g,1} \lambda_k (\alpha_k^2 \sigma_f^2 + (\beta_k^2 + \gamma_k^2) \sigma_g^2).
\end{align*}
We can sum over $k=0$ to $K-1$, and leaving only dominating terms, since $\sum_k \delta_k / \lambda_k^2 = O(1)$ (because $\delta_k / \lambda_k = O(1/k)$ and $\lambda_k = \poly(k)$), we have the theorem.


\subsection{Proof of Corollary \ref{corollary:non_convex_F}}

We first show that with the step-size design in theorem, $\lambda_k = \gamma_k / (2\alpha_k)$ for all $k$. To check this, by design, $\lambda_0 = \gamma_0 / (2\alpha_0)$ and by mathematical induction, 
\begin{align*}
    \frac{T \mu_g}{16} \alpha_k \lambda_k^2 = \frac{T}{32} \frac{c_\gamma}{2 c_\alpha} (k+k_0)^{-2c + a},
\end{align*}
and 
\begin{align*}
    \frac{c_\gamma}{2 c_\alpha} ((k + k_0 + 1)^{a-c} - (k+k_0)^{a-c}) \le \frac{(a-c) c_\gamma}{2 c_\alpha} (k+k_0)^{-1-c+a}.
\end{align*}
As long as $-2c + a \ge -1 - c + a$, or equivalently, $c \le 1$ and $T \ge 32$, it always holds that
\begin{align}
    \lambda_{k+1} = \frac{c_\gamma}{2c_\alpha} (k+k_0 + 1)^{a-c} = \frac{\gamma_{k+1}}{2\alpha_{k+1}}. \label{eq:lambda_math_induction}
\end{align}

Now applying the step-size designs, we obtain the following:
\begin{align}
    \sum_{k=0}^{K-1} \frac{\Exs[\|\grad F(x_k)\|^2]}{(k+k_0)^{a}} &\le O_{\texttt{P}} (1) \cdot \sum_{k} \frac{1}{(k+k_0)^{3a-2c}} + O_{\texttt{P}}(\sigma_f^2) \cdot \sum_{k} \frac{1}{(k+k_0)^{a+c}} \nonumber \\
    &\quad + O_{\texttt{P}}(\sigma_g^2) \cdot \sum_{k} \frac{1}{(k+k_0)^{3c-a}} + O_{\texttt{P}}(1). 
\end{align}
We decide the rates $a,c \in [0,1]$ will be decided differently for different stochasticity. Let $b = a - c$. Note that with the step size deisng, we have $\lambda_k = \gamma_{k} / (2\alpha_k) = \frac{2\lambda_0}{k_0^{a-c}} (k+k_0)^{a-c} = O(k^{b})$. Let $R$ be a random variable uniformly distributed over $\{0,1, ..., K\}$. Note that the left hand side is larger than
\begin{align*}
    \frac{K}{(K+k_0)^a} \sum_{k=1}^{K-1} \frac{1}{K} \Exs[\|\grad F(x_k)\|^2] \ge K^{1-a} \cdot \Exs[\|\grad F(x_R)\|^2]. 
\end{align*}
We consider three regimes:

\paragraph{Stochasticity in both upper-level and lower-level objectives: $\sigma_f^2, \sigma_g^2 > 0$.} In this case, we set $a = 5/7, c = 4/7$, and thus $\lambda_k = k^{1/7}$. The dominating term is $\sigma_g^2 \cdot \sum_k (\gamma_k^2 \lambda_k) = \sum_k O(k^{-1}) = O(\log K)$ and $C_\lambda^2 \cdot \sum_k(\alpha_k \lambda_k^{-2}) = O(\log K)$. From the left-hand side, we have $K^{1-a} = K^{2/7}$. Therefore,
\begin{align*}
    \Exs[\|\grad F(x_R)\|^2] = O\left( \frac{\log K}{K^{2/7}} \right).
\end{align*}

\paragraph{Stochasticity only in the upper-level: $\sigma_f^2 > 0, \sigma_g^2 = 0$.}
In this case, we can take $a = 3/5, c = 2/5$. When $\sigma_g^2 = 0$, the dominating term is $\sigma_f \cdot \sum_k (\alpha_k^2 \lambda_k) = \sum_{k} O(k^{-1}) = O(\log K)$ and $O(C_\lambda^2) \cdot \sum_k (\alpha_k \lambda_k^{-2}) = \sum_{k} O(k^{-1}) = O(\log K)$. Since $K^{1-a} = O(K^{2/5})$, yielding
\begin{align*}
    \Exs[\|\grad F(x_R)\|^2] = O\left( \frac{\log K}{K^{2/5}} \right).
\end{align*}

\paragraph{Deterministic case: $\sigma_f^2 = 0, \sigma_g^2 = 0$.}
Here, we can take $a = 1/3, c = 0$ with a dominating term $\sum_k (\alpha_k \lambda_k^{-2}) = O(\log K)$. Since there is no stochasticity in the algorithm, we have
\begin{align*}
    \|\grad F(x_K)\|^2 = O\left( \frac{\log K}{K^{2/3}} \right).
\end{align*}

\section{Main Results for Algorithm \ref{algo:algo_name2}}
\label{appendix:momentum_method}
We start with a few definitions and additional auxiliary lemmas. We first define the momentum-assisted moving direction of variables. They can be recursively defined as
\begin{align*}
    \tilde{h}_{z}^{k} &:= \nabla_y g(x_k, z_{k}; \phi_{z}^{k}) + (1 - \eta_{k}) \left( \tilde{h}_{z}^{k-1} - \grad_y g(x_{k-1}, z_{k-1}; \phi_{z}^{k}) \right), \\
    \tilde{h}_{fy}^{k} &:= \nabla_y f(x_k, y_{k}; \zeta_{y}^{k}) + (1 - \eta_{k}) \left( \tilde{h}_{fy}^{k-1} - \grad_y f(x_{k-1}, y_{k-1}; \zeta_{y}^{k}) \right), \\
    \tilde{h}_{gy}^{k} &:= \nabla_y g (x_k, y_{k}; \phi_{y}^{k}) + (1 - \eta_{k}) \left( \tilde{h}_{gy}^{k-1} - \grad_y g (x_{k-1}, y_{k-1}; \phi_{y}^{k}) \right), 
\end{align*}
for the inner variable updates, and 
\begin{align*}
    \tilde{h}_{fx}^k &:= \nabla_x f(x_k, y_{k+1}; \zeta_{x}^k) + (1 - \eta_k) \left(\tilde{h}_{fx}^{k-1} - \nabla_x f(x_{k-1}, y_{k}; \zeta_{x}^k) \right), \\
    \ \tilde{h}_{gxy}^k &:= \nabla_{x} g(x_k, y_{k+1}; \phi_{x}^k) + (1 - \eta_k) \left(\tilde{h}_{gxy}^{k-1} - \nabla_x g(x_{k-1}, y_{k}; \phi_{x}^k) \right), \\ 
    \ \tilde{h}_{gxz}^k &:= \nabla_{x} g(x_k, z_{k+1}; \phi_{x}^k) + (1 - \eta_k) \left(\tilde{h}_{gxz}^{k-1} - \nabla_x g(x_{k-1}, z_{k}; \phi_{x}^k) \right).
\end{align*}
for the outer variable update with some proper choice of $\eta_k$. We also define stochastic error terms incurred by random sampling:
\begin{align}
    \tilde{e}_k^x &:= \tilde{h}_{fx}^k + \lambda_k (\tilde{h}_{gxy}^k - \tilde{h}_{gxz}^k) - q_k^x, \nonumber \\
    \tilde{e}_k^y &:= (\tilde{h}_{fy}^k + \lambda_k \tilde{h}_{gy}^k) - q_k^y, \nonumber \\
    \tilde{e}_k^z &:= \tilde{h}_{z}^k - q_k^z, \label{eq:def_ek}
\end{align}
where $q_k^x, q_k^y, q_k^z$ are 
defined in \eqref{eq:qk_def} (we dropped $t$ from subscript since here we consider $T=1$).

\subsection{Additional Auxiliary Lemmas}
The following lemmas are analogous of Lemma \ref{lemma:y_star_contraction}. 

\begin{lemma}
    \label{lemma:y_star_contraction_momentum}
    At every $k$ iteration conditioned on $\mathcal{F}_k$, we have
    \begin{align*}
        \Exs[\|y^*(x_{k+1}) - y^*(x_k)\|^2| \mathcal{F}_k] \le 2\xi^2 l_{*,0}^2 \alpha_k^2 \left( \Exs[\|q_k^x\|^2 | \mathcal{F}_k] + \Exs[\|\tilde{e}_k^x\|^2] \right). 
    \end{align*}
\end{lemma}

\begin{lemma}
    \label{lemma:y_star_lambda_contraction_momentum}
    At every $k$ iteration conditioned on $\mathcal{F}_k$, we have
    \begin{align*}
        \Exs[\|y_{\lambda_{k+1}}^*(x_{k+1}) - y_{\lambda_k}^*(x_k)\|^2| \mathcal{F}_k] \le 4\xi^2 l_{*,0}^2 \alpha_k^2 \left( \Exs[\|q_k^x\|^2 | \mathcal{F}_k] + \Exs[\|\tilde{e}_k^x\|^2] \right) + \frac{8\delta_k^2 l_{f,0}^2}{\lambda_k^4 \mu_g^2}. 
    \end{align*}
\end{lemma}

\subsection{Descent Lemma for Noise Variances}
A major change in the proof is that now we also track the decrease in stochastic error terms. Specifically, we show the following lemmas.
\begin{lemma}
    \label{lemma:descent_stochastic_noise_yz}
    \begin{align*}
        \Exs[\| \tilde{e}_{k+1}^z\|^2] &\le (1 - \eta_{k+1})^2 (1 + 8l_{g,1}^2 \gamma_k^2) \Exs[\|\tilde{e}_k^z\|^2] + 2 \eta_{k+1}^2 \sigma_g^2 \\
        &\qquad + 8 l_{g,1}^2 (1-\eta_{k+1})^2 \left( \xi^2 \alpha_k^2 \Exs[\|q_k^x\|^2] + \xi^2 \alpha_k^2 \Exs[\|\tilde{e}_k^x\|^2] + \gamma_k^2 \Exs[\|q_k^z\|^2] \right), \\
        \Exs[\| \tilde{e}_{k+1}^y\|^2] &\le (1-\eta_{k+1})^2 (1 + 96 l_{g,1}^2 \beta_k^2 ) \Exs[\| \tilde{e}_{k}^y\|^2 ] + 2 \eta_{k+1}^2 (\sigma_f^2 + \lambda_{k+1}^2 \sigma_g^2) + 12 \delta_k^2 \sigma_g^2 \\
    &\qquad + 96 l_{g,1}^2 (1-\eta_{k+1})^2 \beta_k^2 (\xi^2 \|q_k^x\|^2 + \xi^2 \|\tilde{e}_k^x\|^2 + \|q_k^y\|^2).
    \end{align*}
\end{lemma}

\begin{lemma}
    \label{lemma:descent_stochastic_noise_x}
    \begin{align*}
        \Exs[\| \tilde{e}_{k+1}^x\|^2] &\le (1 - \eta_{k+1})^2 (1 + 240 l_{g,1}^2 \xi^2 \beta_k^2) \Exs[\|\tilde{e}_{k}^x\|^2 ] + 6\eta_{k+1}^2 (\sigma_f^2 + \lambda_{k+1}^2 \sigma_g^2) + 80 \delta_k^2 \sigma_g^2 \\
        &\quad + 240 l_{g,1}^2 (1-\eta_{k+1})^2 \lambda_k^2 \left(\xi^2 \alpha_k^2 \|q_k^x\|^2 + \alpha_k^2 (\|q_k^y\|^2 + \|\tilde{e}_k^y\|^2) + \gamma_k^2 (\|q_k^z\|^2 + \|\tilde{e}_k^z\|^2) \right). 
    \end{align*}
\end{lemma}
Equipped with these lemmas, we can now proceed as previously in the main proof for Algorithm \ref{algo:algo_name}. 

\subsection{Descent Lemma for $z_k$ towards $y_k^*$}
\begin{lemma}
    \label{lemma:z_descent_1}
    If $\gamma_k \mu_g < 1/8$, then
    \begin{align*}
        \Exs[\|z_{k+1} - y_{k+1}^*\|^2 | \mF_k] &\le (1 + \gamma_k \mu_g / 4) \Exs[\|z_{k+1} - y_k^*\|^2 | \mF_k] \\
        &\quad + O\left( \frac{\xi^2 \alpha_k^2 l_{*,0}^2}{\gamma_k \mu_g} \right) \cdot (\Exs[\|q_k^x\|^2 | \mF_k] + \Exs[\|\tilde{e}_k^x\|^2 | \mF_k]). 
    \end{align*}
\end{lemma}
\begin{proof}
    As before, we can decompose $\|z_{k+1} - y_{k+1}^*\|^2$ as
    \begin{align*}
        \|z_{k+1} - y_{k+1}^*\|^2 &= \|z_{k+1} - y_k^*\|^2 + \|y_{k+1}^* - y_k^*\|^2 - 2\vdot{z_{k+1}-y_k^*}{y_{k+1}^* - y_k^*} \\
        &\le \|z_{k+1} - y_k^*\|^2 + \left(1 + \frac{1}{8\gamma_k \mu_g} \right) \|y_{k+1}^* - y_k^*\|^2 + 4\gamma_k \mu_g \|z_{k+1}-y_k^*\|^2,
    \end{align*}
    where we used a general inequality $|\vdot{a}{b}| \le c\|a\|^2 + \frac{1}{4c} \|b\|^2$.
    We can apply Lemma \ref{lemma:y_star_contraction_momentum} for $\|y_{k+1}^* - y_k^*\|^2$, yielding the lemma. 
\end{proof}

\begin{lemma}
    \label{lemma:z_descent_2}
    If $\gamma_k \le 1 / (16 l_{g,1})$, then
    \begin{align*}
        \Exs[\|z_{k+1} - y_{k}^*\|^2 | \mF_k] &\le (1 - \gamma_k \mu_g / 2) \Exs[\|z_{k} - y_k^*\|^2 | \mF_k] - \frac{\gamma_k}{l_{g,1}} \|q_k^z\|^2 + O\left(\frac{\gamma_k}{\mu_g} \right) \Exs[\|\tilde{e}_k^z\|^2 | \mF_k]). 
    \end{align*}
\end{lemma}
\begin{proof}
    Note that
    \begin{align*}
        \|z_{k+1} - y_k^*\|^2 &= \|z_k - y_k^*\|^2 + \gamma_k^2 \|\tilde{h}_z^k\|^2 - 2\gamma_k \vdot{\tilde{h}_z^k}{z_k - y_k^*} \\
        &\le \|z_k - y_k^*\|^2 + 2\gamma_k^2 (\|q_k^z\|^2 + \|\tilde{e}_k^z\|^2) - 2\gamma_k \vdot{q_k^z}{z_k - y_k^*} - 2\gamma_k \vdot{\tilde{e}_k^z}{z_k - y_k^*}.
    \end{align*}
    Since $q_k^z = \grad_y g(x_k, z_k)$ by definition, by coercivity and co-coercivity of strongly-convex functions, we have
    \begin{align*}
        \left( \mu_g \|z_k - y_k^*\|^2, \frac{1}{l_{g,1}}\|q_k^z\|^2 \right) \le \vdot{q_k^z}{z_k - y_k^*},  
    \end{align*}
    and thus, given $\gamma_k \le 1/(16l_{g,1})$, we have
    \begin{align*}
        \Exs[z_{k+1} - y_k^*\|^2 | \mF_k] \le (1 - 3 \gamma_k \mu_g/4) \Exs[\|z_k - y_k^*\|^2 | \mF_k] - \frac{\gamma_k}{l_{g,1}} \|q_k^z\|^2 + 2\gamma_k^2 \Exs[\|\tilde{e}_k^z\|^2 | \mF_k] - 2\gamma_k \vdot{\tilde{e}_k^z}{z_k - y_k^*}.
    \end{align*}
    Finally, we can use general inequality $|\vdot{a}{b}| \le c\|a\|^2 + \frac{1}{4c} \|b\|^2$ to get
    \begin{align*}
        -2\gamma_k \vdot{\tilde{e}_k^z}{z_k - y_k^*} \le \frac{\gamma_k \mu_g}{4} \|z_k - y_k^*\|^2 + \frac{4\gamma_k}{\mu_g} \|\tilde{e}_k^z\|^2. 
    \end{align*}
    Plugging this back, with $\gamma_k^2 \ll \frac{\gamma_k}{\mu_g}$, we get the lemma. 
\end{proof}

\subsection{Descent Lemma for $y_k$ towards $y_{\lambda,k}^*$}
\begin{lemma}
    \label{lemma:y_descent_1}
    If $\beta_k \mu_g < 1/8$, then
    \begin{align*}
        \Exs[\|y_{k+1} - y_{\lambda, k+1}^*\|^2 | \mF_k] &\le (1 + \beta_k \mu_g / 4) \Exs[\|y_{k+1} - y_{\lambda,k}^*\|^2 | \mF_k] \\
        &\quad + O\left( \frac{\xi^2 \alpha_k^2 l_{*,0}^2}{\beta_k \mu_g} \right) \cdot (\Exs[\|q_k^x\|^2 | \mF_k] + \Exs[\|\tilde{e}_k^x\|^2 | \mF_k]) + O\left(\frac{\delta_k^2 l_{f,0}^2}{\lambda_k^4 \mu_g^2} \right). 
    \end{align*}
\end{lemma}
\begin{proof}
    As before, we can decompose $\|y_{k+1} - y_{\lambda, k+1}^*\|^2$ as
    \begin{align*}
        \|y_{k+1} - y_{\lambda,k+1}^*\|^2 &= \|y_{k+1} - y_{\lambda,k}^*\|^2 + \|y_{k+1}^* - y_{\lambda,k}^*\|^2 - 2\vdot{y_{k+1}-y_{\lambda,k}^*}{y_{\lambda,k+1}^* - y_{\lambda,k}^*} \\
        &\le \|y_{k+1} - y_{\lambda,k}^*\|^2 + \left(1 + \frac{1}{8\beta_k \mu_g} \right) \|y_{\lambda,k+1}^* - y_{\lambda,k}^*\|^2 + 4\beta_k \mu_g \|y_{k+1}-y_{\lambda,k}^*\|^2,
    \end{align*}
    where we used a general inequality $|\vdot{a}{b}| \le c\|a\|^2 + \frac{1}{4c} \|b\|^2$.
    We can apply Lemma \ref{lemma:y_star_lambda_contraction_momentum} for $\|y_{\lambda,k+1}^* - y_{\lambda,k}^*\|^2$, 
    since $\beta_k \mu_g \le 1/16$, we get the lemma.
\end{proof}

\begin{lemma}
    \label{lemma:y_descent_2}
    If $\beta_k \le 1 / (16 l_{g,1})$, then
    \begin{align*}
        \Exs[\|y_{k+1} - y_{\lambda,k}^*\|^2 | \mF_k] &\le (1 - \beta_k \mu_g / 2) \Exs[\|y_{k} - y_{\lambda,k}^*\|^2 | \mF_k] - \frac{\alpha_k}{\lambda_k l_{g,1}} \|q_k^y\|^2 + O\left(\frac{\alpha_k^2}{\mu_g \beta_k} \right) \Exs[\|\tilde{e}_k^y\|^2 | \mF_k]). 
    \end{align*}
\end{lemma}
\begin{proof}
    Note that
    \begin{align*}
        \|y_{k+1} - y_{\lambda,k}^*\|^2 &= \|y_k - y_{\lambda,k}^*\|^2 + 2 \alpha_k^2 (\|q_k^y\|^2 + \|\tilde{e}_k^y\|^2) - 2\alpha_k \vdot{q_k^y}{y_k - y_{\lambda,k}^*} - 2\alpha_k \vdot{\tilde{e}_k^y}{y_k - y_{\lambda,k}^*},
    \end{align*}
    where we used $y_{k+1} - y_k = q_k^y + \tilde{e}_k^y$. Since $q_k^y = \grad_y \mL_{\lambda_k}$ by definition, again by coercivity and co-coercivity of strongly-convex $\mL_{\lambda_k}(x_k, \cdot)$, we have
    \begin{align*}
        \max \left( \frac{\lambda_k \mu_g}{2} \|y_k - y_{\lambda,k}^*\|^2, \frac{1}{l_{f,1} + \lambda_k l_{g,1}} \|q_k^y\|^2 \right) \le \vdot{q_k^y}{y_k - y_{\lambda,k}^*},  
    \end{align*}
    and thus, given $\alpha_k \lambda_k \le 1/(16 l_{g,1})$ and $l_{f,1} \le \lambda_k l_{g,1}$, we have
    \begin{align*}
        \Exs[y_{k+1} - y_{\lambda,k}^*\|^2 | \mF_k] &\le (1 - 3 \beta_k \mu_g/4) \Exs[\|z_k - y_k^*\|^2 | \mF_k] - \frac{\alpha_k}{\lambda_k l_{g,1}} \|q_k^y\|^2 \\
        &\quad + 2\alpha_k^2 \Exs[\|\tilde{e}_k^y\|^2 | \mF_k] - 2\alpha_k \Exs[\vdot{\tilde{e}_k^y}{y_k - y_{\lambda,k}^*} | \mF_k].
    \end{align*}
    Finally, we can use general inequality $|\vdot{a}{b}| \le c\|a\|^2 + \frac{1}{4c} \|b\|^2$ to get
    \begin{align*}
        -2\alpha_k \vdot{\tilde{e}_k^y}{y_k - y_{\lambda,k}^*} \le \frac{\beta_k \mu_g}{4} \|y_k - y_{\lambda,k}^*\|^2 + \frac{4\alpha_k^2}{\beta_k \mu_g} \|\tilde{e}_k^y\|^2. 
    \end{align*}
    Plugging this back, with $\beta_k \mu_g \ll 1$, we get the lemma. 
\end{proof}

\subsection{Descent Lemma for $F(x_{k})$}
\begin{lemma}
    If $\xi \alpha_k l_{F,1} < 1$, then
    \begin{align*}
        \Exs[F(x_{k+1}) - F(x_k) | \mF_k] &\le -\frac{\xi \alpha_k}{4} \|\grad F(x_k)\|^2 - \frac{\xi \alpha_k}{4} \Exs[\|q_k^x\|^2| \mF_k] + 2 \xi \alpha_k \cdot \Exs[\|\tilde{e}_k^x\|^2 | \mF_k] \\
        &\quad + \frac{3 \xi \alpha_k}{2} \left(4 l_{g,1}^2 \lambda_k^2 \|y_{k+1} - y_{\lambda,k}^*\|^2 + l_{g,1}^2 \lambda_k^2 \|z_{k+1} - y_k^*\|^2 + C_\lambda^2 / \lambda_k^2 \right). 
    \end{align*}
\end{lemma}
\begin{proof}
    Using the smoothness of $F$, 
    \begin{align*}
        F(x_{k+1}) - F(x_k) &\le \vdot{\grad F(x_k)}{x_{k+1} - x_k} + \frac{l_{F,1}}{2} \|x_{k+1} - x_k\|^2 .
    \end{align*}
    Note that $x_{k+1} - x_k = \xi \alpha_k (q_k^x + \tilde{e}_k^x)$, and thus 
    \begin{align*}
        &F(x_{k+1}) - F(x_k) \le -\xi \alpha_k \vdot{\grad_x F(x_k)}{q_k^x} -\xi \alpha_k \vdot{\grad_x F(x_k)}{\tilde{e}_k^x} + \frac{l_{F,1}}{2} \|x_{k+1} - x_k\|^2  \\
        &\le -\frac{\xi \alpha_k}{2} (\|\grad F(x_k)\|^2 + \|q_k^x\|^2 - \| \grad F(x_k) - q_k^x \|^2) -\xi \alpha_k \vdot{\grad_x F(x_k)}{\tilde{e}_k^x} + \xi^2\alpha_k^2 l_{F,1} (\|q_k^x\|^2 + \| \tilde{e}_k^x\|^2).
    \end{align*}
    Using $|\vdot{a}{b}| \le c\|a\|^2 + \frac{1}{4c}\|b\|^2$, we have
    \begin{align*}
        -\xi\alpha_k \vdot{\grad F(x_k)}{\tilde{e}_k^x} \le \frac{\xi \alpha_k}{4} \|\grad_x F(x_k)\|^2 + \xi\alpha_k \|\tilde{e}_k^x\|^2.
    \end{align*}
    Finally, recall \eqref{eq:fxfx1}. Using $(a+b+c)^2 \le 3(a^2 + b^2 + c^2)$, we have
    \begin{align*}
        \|\grad F(x_k) - q_k^x\|^2 \le 3(4l_{g,1}^2\lambda_k^2 \|y_{k+1} - y_{\lambda,k}^*\|^2 + l_{g,1}^2 \lambda_k^2 \|z_{k+1} - y_k^*\|^2 + C_\lambda^2 / \lambda_k^2). 
    \end{align*}
    Combining all, with $\xi \alpha_k l_{F,1} < 1$, we get the lemma. 
\end{proof}

\subsection{Decrease in Potential Function}
Define the potential function $\mathbb{V}_k$ as the following:
\begin{align*}
    \mathbb{V}_k := F(x_k) + l_{g,1} \lambda_k \|y_k - y_{\lambda, k}^*\|^2 + \frac{l_{g,1} \lambda_k }{2} \|z_k - y_k^*\|^2 + \frac{1}{c_\eta l_{g,1}^2 \gamma_{k-1}} \left(\frac{\|\tilde{e}_k^x\|^2}{\lambda_k} + \frac{\|\tilde{e}_k^y\|^2}{\lambda_k} + \lambda_k \|\tilde{e}_k^z\|^2 \right),
\end{align*}
with some absolute constant $c_\eta > 0$. We bound the difference in potential function:
\begin{align*}
    \Exs[ \mathbb{V}_{k+1} - \mathbb{V}_k | \mF_k] &\le -\frac{\xi \alpha_k}{4} \|\grad F(x_k)\|^2 - \frac{\xi \alpha_k}{4} \Exs[\|q_k^x\|^2 | \mF_k] + \frac{\xi\alpha_k}{2} \frac{3 C_\lambda^2}{\lambda_k^2} \\
    &+ l_{g,1} \lambda_k \underbrace{\left( \left(1 + \frac{\delta_k}{\lambda_k} \right) \|y_{k+1} - y_{\lambda,k+1}^*\|^2 + 6 \xi\alpha_k\lambda_k l_{g,1} \|y_{k+1} - y_{\lambda,k}^*\|^2 - \|y_k - y_{\lambda,k}^*\|^2 \right)}_{(i)} \\
    &+ \frac{l_{g,1} \lambda_k}{2} \underbrace{\left( \left(1 + \frac{\delta_k}{\lambda_k} \right) \|z_{k+1} - y_{k+1}^*\|^2 + 3 \xi\alpha_k\lambda_k l_{g,1} \|z_{k+1} - z_{k}^*\|^2 - \|z_k - y_{k}^*\|^2 \right)}_{(ii)} \\
    &+ \frac{1}{c_\eta l_{g,1}^2 } \underbrace{\left( \frac{\Exs[\|\tilde{e}_{k+1}^x\|^2 | \mF_k]}{\gamma_k \lambda_k} - \frac{\Exs[ \|\tilde{e}_{k}^x\|^2 | \mF_k]}{\gamma_{k-1} \lambda_k} \right)}_{(iii)} + 2 \xi \alpha_k \Exs[\|\tilde{e}_k^x\|^2 | \mF_k]  \\
    &+ \frac{1}{c_\eta l_{g,1}^2} \underbrace{\left( \frac{\Exs[\|\tilde{e}_{k+1}^y\|^2 | \mF_k]}{\gamma_k \lambda_k} - \frac{\Exs[ \|\tilde{e}_{k}^y\|^2 | \mF_k]}{\gamma_{k-1} \lambda_k} \right)}_{(iv)} + \frac{\lambda_k}{c_\eta l_{g,1}^2} \underbrace{\left( \frac{\Exs[\|\tilde{e}_{k+1}^z\|^2 | \mF_k]}{\gamma_k} - \frac{\Exs[ \|\tilde{e}_{k}^z\|^2 | \mF_k]}{\gamma_{k-1}} \right)}_{(v)}. 
\end{align*}
Using Lemmas \ref{lemma:y_descent_1}, \ref{lemma:y_descent_2}, \ref{lemma:z_descent_1} and \ref{lemma:z_descent_2}, given that $\delta_k/\lambda_k < \mu_g \beta_k / 8$, $(i)$ and $(ii)$ are bounded by
\begin{align*}
    (i) &\le -\frac{\mu_g \beta_k}{ 8 }\|y_{k} - y_{\lambda,k}^*\|^2 - \frac{\alpha_k}{\lambda_k l_{g,1}} \|q_k^y\|^2 + O\left( \frac{\xi^2 \alpha_k^2 l_{*,0}^2 }{\beta_k \mu_g} \Exs[\|q_k^x\|^2 + \|\tilde{e}_k^x\|^2 | \mathcal{F}_k] + \frac{\delta_k^2 l_{f,0}^2 }{\lambda_k^4 \mu_g^3 } + \frac{\alpha_k^2}{\beta_k \mu_g} \Exs[\|\tilde{e}_k^y\|^2 | \mathcal{F}_k] \right), \\
    (ii) &\le -\frac{\mu_g \gamma_k}{ 8} \|z_{k} - y_k^*\|^2 - \frac{\gamma_k}{l_{g,1}} \|q_k^z\|^2 + O\left( \frac{\xi^2 \alpha_k^2 l_{*,0}^2 }{\gamma_k \mu_g} \Exs[\|q_k^x\|^2 + \|\tilde{e}_k^x\|^2 | \mathcal{F}_k] + \frac{\gamma_k}{\mu_g} \Exs[\|\tilde{e}_k^z\|^2 | \mathcal{F}_k] \right), 
\end{align*}
We can use Lemma \ref{lemma:descent_stochastic_noise_x} to bound $(iii), (iv)$ and $(v)$. Using the step-size condition given in \eqref{eq:step_size_theorem_momentum_b}, we have
\begin{align*}
    \frac{(1-\eta_{k+1} )}{\gamma_k} - \frac{1}{\gamma_{k-1}} &= \frac{\frac{\gamma_{k-1} - \gamma_{k}}{\gamma_{k-1}} - \eta_{k+1} }{\gamma_k} \le \frac{-\eta_{k+1}}{2 \gamma_k}. 
\end{align*}
Note that by the same step-size condition, $\eta_{k+1} \gg O(l_{g,1}^2 \gamma_k^2)$, and thus, 
\begin{align*}
    (iii) &\le -\frac{\eta_{k+1}}{2 \gamma_k \lambda_k} \Exs[\|\tilde{e}_k^x\|^2 | \mF_k] + O(\sigma_f^2) \cdot \frac{\eta_{k+1}^2}{\lambda_k \gamma_k} + O(\sigma_g^2) \cdot \left(\frac{\eta_{k+1}^2 \lambda_k}{\gamma_k} + \frac{\delta_k^2}{\gamma_k \lambda_k} \right) \\
    &\quad + O(l_{g,1}^2) \cdot \left( \xi^2 \alpha_k \|q_k^x\|^2 + \alpha_k(\|q_k^y\|^2 + \|\tilde{e}_k^y\|^2) + \gamma_k \lambda_k (\|q_k^z\|^2 + \|\tilde{e}_k^z\|^2) \right). 
\end{align*}
Similarly, we can use Lemma \ref{lemma:descent_stochastic_noise_yz} and show that
\begin{align*}
    (iv) &\le -\frac{\eta_{k+1}}{2 \gamma_k \lambda_k} \Exs[\|\tilde{e}_k^y\|^2 | \mF_k] + O(\sigma_f^2) \cdot \frac{\eta_{k+1}^2}{\lambda_k \gamma_k} + O(\sigma_g^2) \cdot \left(\frac{\eta_{k+1}^2 \lambda_k}{\gamma_k} + \frac{\delta_k^2}{\gamma_k \lambda_k} \right) \\
    &\quad + O(l_{g,1}^2) \alpha_k \cdot \left( \xi^2 \|q_k^x\|^2 + \xi^2 \|\tilde{e}_k^x\|^2 + \|q_k^y\|^2  \right), \\
    (v) &\le -\frac{\eta_{k+1}}{2 \gamma_k} \Exs[\|\tilde{e}_k^z\|^2 | \mF_k] + O(\sigma_g^2) \cdot \frac{\eta_{k+1}^2}{\gamma_k}  + O(l_{g,1}^2) \cdot \left( \frac{\xi^2 \alpha_k^2}{\gamma_k} \|q_k^x\|^2 + \frac{\xi^2 \alpha_k^2}{\gamma_k} \|\tilde{e}_k^x\|^2 + \gamma_k \|q_k^z\|^2  \right). 
\end{align*}

Plugging inequalities for $(i)-(v)$ back and arranging terms, we get
\begin{align*}
    \Exs[ \mathbb{V}_{k+1} - \mathbb{V}_k | \mF_k] &\le -\frac{\xi \alpha_k}{4} \|\grad F(x_k)\|^2 - \frac{\xi \alpha_k}{4} \Exs[\|q_k^x\|^2| \mF_k] + \frac{3 C_{\lambda}^2}{2} \frac{\xi \alpha_k}{\lambda_k^2} - l_{g,1}\lambda_k (i) + \frac{l_{g,1} \lambda_k}{2} (ii) \\
    &\quad + \frac{1}{c_\eta l_{g,1}^2} (iii) + 2\xi \alpha_k \Exs[\|\tilde{e}_k^x\|^2 | \mF_k] + \frac{1}{c_{\eta} l_{g,1}^2} (iv) + \frac{\lambda_k}{c_\eta l_{g,1}^2} (v) \\
    &\le -\frac{\xi \alpha_k}{4} \|\grad F(x_k)\|^2 - \frac{\lambda_k l_{g,1} \mu_g \beta_k}{4} \|y_k - y_{\lambda,k}^*\|^2 - \frac{\lambda_k l_{g,1} \mu_g \gamma_k}{4}\|z_k - y_k^*\|^2 \\
    &\quad - \frac{\xi \alpha_k}{4} \Exs[\|q_k^x\|^2 | \mF_k] \left(1 - O(\xi l_{g,1} l_{*,0}^2 / 
    \mu_g) - O(\xi c_{\eta}^{-1}) \right) \\
    &\quad - \alpha_k \Exs[\|q_k^y\|^2 |\mF_k] \left(1 - O(c_{\eta}^{-1}) \right) - \frac{\gamma_k \lambda_k}{2} \Exs[\|q_k^z\|^2 |\mF_k] \left(1 - O(c_{\eta}^{-1}) \right) \\
    &\quad + O\left( \frac{C_\lambda^2 \xi \alpha_k}{\lambda_k^2} + \frac{l_{f,0}^2 l_{g,1} \delta_k^2}{\mu_g^3 \lambda_k^3} \right) + \text{noise variance terms},
\end{align*}
where noise variance terms are
\begin{align*}
    \text{noise terms} &= -\frac{\Exs[\|\tilde{e}_k^x\|^2 | \mF_k]}{c_{\eta} l_{g,1}^2} \left( \frac{\eta_{k+1}}{2\gamma_k\lambda_k} - O\left(l_{g,1}^2 \xi^2 \alpha_k \right) - (c_\eta \xi \alpha_k l_{g,1}^2) \right) \\
    &\quad - \frac{\Exs[\|\tilde{e}_k^y\|^2 | \mF_k]}{l_{g,1}^2 c_{\eta}} \left( \frac{\eta_{k+1}}{2\gamma_k \lambda_k} - O(l_{g,1}^2) \alpha_k - O(c_\eta l_{g,1}^3 / \mu_g) \alpha_k \right) \\
    &\quad - \frac{\lambda_k \Exs[\|\tilde{e}_k^z\|^2 | \mF_k]}{l_{g,1}^2 c_{\eta}} \left( \frac{\eta_{k+1}}{2\gamma_k} - O(l_{g,1}^2) \gamma_k - O(c_\eta l_{g,1}^3 / \mu_g) \gamma_k \right) \\
    &\quad + \frac{1}{c_\eta l_{g,1}^2} 
    \left( O(\sigma_f^2) \cdot \frac{\eta_{k+1}^2}{\lambda_k \gamma_k} + O(\sigma_g^2) \cdot \left( \frac{\eta_{k+1}^2 \lambda_k}{\gamma_k} + \frac{\delta_k^2}{\gamma_k \lambda_k} \right) \right). 
\end{align*}
For the all squared terms, with careful design of step-sizes, we can make the coefficient negative. Specifically, we need
\begin{align*}
    &\xi l_{g,1} l_{*,0}^2 / \mu_g \ll 1, \ c_\eta \gg 1,
\end{align*}
to negate $q_k^{(\cdot)}$ terms, and 
\begin{align*}
    1 > \eta_{k+1} \gg  c_{\eta} \gamma_k^2 (l_{g,1}^3/\mu_g),
\end{align*}
to suppress noise variance terms, as required in our step-size rules \eqref{eq:step_size_theorem_momentum}. Then, we can simplify the bound for the potential function difference:
\begin{align*}
    \Exs[\mathbb{V}_{k+1} - \mathbb{V}_k | \mF_k] &\le -\frac{\xi \alpha_k}{4} \|\grad F(x_k)\|^2 + O(\xi C_\lambda^2) \cdot \frac{\alpha_k}{\lambda_k^{2}} + O(l_{f,0}^2 l_{g,1}/\mu_g^3) \cdot \frac{\delta_k^2}{ \lambda_k^{3}} \\
    &\quad + \frac{1}{c_\eta l_{g,1}^2} \left( O(\sigma_f^2) \cdot \frac{\eta_{k+1}^2}{\lambda_k \gamma_k} + O(\sigma_g^2) \cdot \left( \frac{\eta_{k+1}^2 \lambda_k}{\gamma_k} + \frac{\delta_k^2}{\gamma_k \lambda_k} \right) \right).
\end{align*}

\paragraph{Proof of Theorem \ref{theorem:general_momentum}}
Summing the above over all $k = 0$ to $K-1$, using $\delta_k/ \lambda_k = O(1/k)$ and $1/\lambda_k, \delta_k / \gamma_k = o(1)$, we obtain Theorem \ref{theorem:general_momentum}.

\subsection{Proof of Corollary \ref{corollary:non_convex_momentum}}
Using the step-sizes specified in \eqref{eq:step_size_momentum_detail}, since $\lambda_k = \gamma_k / 2\alpha_k \asymp k^{a-c}$, $\delta_k \asymp k^{a-c-1}$. As long as $a-c-1 < -c$, which is satisfied if $a < 1$, we have $\delta_k / \gamma_k = o(1)$. We can also check that 
\begin{align*}
    \frac{\delta_k}{\lambda_k} \le (k+k_0+1)^{-1} < \frac{\mu_g \beta_k}{8} = \frac{(k+k_0)^{-c}}{k_0^{1-c}},
\end{align*}
as long as and $c < 1$. Given the above, we have
\begin{align*}
    \sum_{k=0}^{K-1} \frac{\Exs[\|\grad F(x_k)\|^2]}{(k+k_0)^{a}} &\le O_{\texttt{P}} (1) \cdot \sum_{k} \frac{1}{(k+k_0)^{3a-2c}} + O_{\texttt{P}}(\sigma_f^2) \cdot \sum_{k} \frac{1}{(k+k_0)^{-2c-a}} \nonumber \\
    &\quad + O_{\texttt{P}}(\sigma_g^2) \cdot \sum_{k} \frac{1}{(k+k_0)^{-4c+a}} + O_{\texttt{P}}(1). 
\end{align*}
Again, we consider three regimes:

\paragraph{Stochasticity in both upper-level and lower-level objectives: $\sigma_f^2, \sigma_g^2 > 0$.} In this case, we set $a = 3/5, c = 2/5$, and thus $\lambda_k \asymp k^{1/5}$, which yields
\begin{align*}
    \Exs[\|\grad F(x_R)\|^2] \asymp \frac{\log K}{K^{2/5}}.
\end{align*}

\paragraph{Stochasticity only in the upper-level: $\sigma_f^2 > 0, \sigma_g^2 = 0$.}
In this case, we can take $a = 2/4, c = 1/4$, and thus $\lambda_k \asymp k^{1/4}$, which yields
\begin{align*}
    \Exs[\|\grad F(x_R)\|^2] \asymp \frac{\log K}{K^{2/4}}.
\end{align*}

\paragraph{Deterministic case: $\sigma_f^2 = 0, \sigma_g^2 = 0$.}
Here, we can take $a = 1/3, c = 0$ and since there is no stochasticity in the algorithm, we have
\begin{align*}
    \|\grad F(x_K)\|^2 \asymp \frac{\log K}{K^{2/3}}.
\end{align*}

\section{Deferred Proofs for Lemmas}
\label{appendix:deferred_proof}

\subsection{Proofs for Main Lemmas}
\subsubsection{Proof of Lemma \ref{lemma:l_star_lambda_approximate}}
\begin{proof}
    Let $y^*_\lambda(x) := \arg\min_y \mL_{\lambda}(x,y)$. Note that $\grad_y \mL_{\lambda} (x,y^*_\lambda(x)) = 0$, and thus
    \begin{align*}
        \grad \mL_{\lambda}^* (x) &= \grad_x \mL_{\lambda} (x, y_{\lambda}^*(x)) + \grad_x y_{\lambda}^*(x)^\top \grad_y \mL_{\lambda} (x, y_{\lambda}^*(x)) =  \grad_x \mL_{\lambda} (x, y_{\lambda}^*(x)).
    \end{align*}
    To compare this to $\grad F(x)$, we can invoke Lemma \ref{lemma:relation_Lagrangian_F} which gives
    \begin{align*}
        &\| \grad F(x) - \grad_x \mL_\lambda(x,y_{\lambda}^*(x) )  \| \\
        &\qquad \le 2 (l_{g,1}/\mu_g) \|y_{\lambda}^*(x) - y^*(x)\| \left(l_{f,1} + \lambda \cdot \min(2l_{g,1}, l_{g,2} \|y^*(x) -  y_{\lambda}^*(x)\|)\right).
    \end{align*}
    From Lemma \ref{lemma:y_star_lagrangian_continuity}, we use $\|y_{\lambda}^*(x) - y^*(x)\| \le \frac{2l_{f,0}}{\lambda\mu_g}$, and get
    \begin{align*}
        \| \grad F(x) - \grad_x \mL_\lambda(x,y_{\lambda}^*(x) )  \| \le \frac{1}{\lambda} \cdot \frac{4l_{f,0} l_{g,1}}{\mu_g^2} \left( l_{f,1} + \frac{2  l_{f,0} l_{g,2}}{\mu_g} \right). 
    \end{align*}
\end{proof}

\subsubsection{Proof of Lemma \ref{lemma:y_star_lagrangian_continuity}}
\begin{proof}
    Note that $\mL_{\lambda}(x,y)$ is at least $\frac{\lambda\mu_g}{2}$ strongly-convex in $y$ once $\lambda \ge 2 l_{f,1} \mu_g$. To see this,
    \begin{align*}
        \mL_{\lambda}(x,y) = f(x,y) + \lambda (g(x,y) - g^*(x)),
    \end{align*}
    which is at least $-l_{f,1} + \lambda \mu_g$-strongly convex in $y$. If $\lambda > 2l_{f,1} / \mu_g$, this implies at least $\lambda \mu_g / 2$ strong-convexity of $\mL_{\lambda} (x,y)$ in $y$. 

    By the optimality condition at $y_{\lambda_1}^*(x_1)$ with $x_1, \lambda_1$, we have
    \begin{align*}
        \grad_y f(x_1, y_{\lambda_1}^*(x_1) ) + \lambda_1 \grad_y g(x_1, y_{\lambda_1}^*(x_1) ) = 0,
    \end{align*}
    which also implies that $\|g(x_1, y_{\lambda_1}^*(x_1) )\| \le l_{f,0} / \lambda_1$. Observe that
    \begin{align*}
        &\grad_y f(x_2, y_{\lambda_1}^*(x_1)) + \lambda_2 \grad_y g(x_2, y_{\lambda_1}^*(x_1)) \\
        &\quad = (\grad_y f(x_2, y_{\lambda_1}^*(x_1)) - \grad_y f(x_1, y_{\lambda_1}^*(x_1))) + \grad_y f(x_1, y_{\lambda_1}^*(x_1)) \\
        &\qquad + \lambda_2 (\grad_y g(x_2, y_{\lambda_1}^*(x_1)) - \grad_y g(x_1, y_{\lambda_1}^*(x_1))) + \lambda_2 \grad_y g(x_1, y_{\lambda_1}^*(x_1)) \\
        &\quad = (\grad_y f(x_2, y_{\lambda_1}^*(x_1)) - \grad_y f(x_1, y_{\lambda_1}^*(x_1))) + \lambda_2 (\grad_y g(x_2, y_{\lambda_1}^*(x_1)) - \grad_y g(x_1, y_{\lambda_1}^*(x_1))) \\
        &\qquad + (\lambda_2 - \lambda_1) \grad_y g(x_1, y_{\lambda_1}^*(x_1)),
    \end{align*}
    where in the last equality, we applied the optimality condition for $y_{\lambda_1}^*(x_1)$. Then applying the Lipschitzness of $\grad_y f$ and $\grad_y g$ in $x$, we have
    \begin{align*}
        \|\grad_y f(x_2, y_{\lambda_1}^*(x_1)) + \lambda_2 \grad_y g(x_2, y_{\lambda_1}^*(x_1))\| &\le l_{f,1} \|x_1 - x_2\| + l_{g,1} \lambda_2 \|x_2 - x_1\| + (\lambda_2 - \lambda_1) \frac{l_{f,0}}{\lambda_1}. 
    \end{align*}
    Since $\mL_{\lambda_2} (x_2, y)$ is $\lambda_2\mu_g/2$-strongly convex in $y$, from the coercivity property of strongly-convex functions, along with the optimality condition with $y^*_{\lambda_2}(x_2)$, we have
    \begin{align*}
        \frac{\lambda_2 \mu_g}{2} \|y_{\lambda_1}^*(x_1) - y^*_{\lambda_2}(x_2) \| \le \|\grad_y \mL_{\lambda_2} (x_2,y_{\lambda_1}^*(x_1))\| \le (l_{f,1}+\lambda_2 l_{g,1}) \|x_1 - x_2\| + \frac{\lambda_2 - \lambda_1}{\lambda_1} l_{f,0}.
    \end{align*}
    Dividing both sides by $(\lambda_2\mu_g/2)$ concludes the first part of the proof. 
    Note that $y^*(x) = \lim_{\lambda \rightarrow \infty} y_\lambda^*(x)$. Thus, for any $x$ and finite $\lambda \ge 2l_{f,1} / \mu_g$,
    \begin{align*}
        \|y^*_{\lambda} (x) - y^*(x)\| \le \frac{2l_{f,0}}{\lambda\mu_g}. 
    \end{align*}
\end{proof}

\subsection{Proofs for Auxiliary Lemmas}
\subsubsection{Proof of Lemma \ref{lemma:outer_F_smooth}}
\begin{proof}
    The proof can be also found in Lemma 2.2 in \cite{ghadimi2018approximation}. We provide the proof for the completeness. Recall that $\grad F(x)$ is given by
    \begin{align*}
        \grad F(x) = \grad_x f(x,y^*(x)) - \grad_{xy}^2 g(x,y^*(x)) \grad_{yy}^2 g(x,y^*(x))^{-1} \grad_y f(x,y^*(x)).
    \end{align*}
    Using the smoothness of functions and Hessian-continuity of $g$ in assumptions, for any $x_1, x_2 \in X$, we get
    \begin{align*}
        \|\grad F(x_1) - \grad F(x_2)\| &\le \left(l_{f,1}  + \frac{l_{f,0}}{\mu_g} l_{g,2} + \frac{l_{g,1}}{\mu_g} l_{g,1} \right) (\|x_1 - x_2\| + \|y^*(x_1) - y^*(x_2)\|) \\
        &\quad + l_{g,1} l_{f,0} \| \grad_{yy}^2 g(x_1, y^*(x_1))^{-1} -  \grad_{yy}^2 g(x_2, y^*(x_2))^{-1} \| \\
        &\le \left(l_{f,1}  + \frac{l_{f,0}}{\mu_g} l_{g,2} + \frac{l_{g,1}}{\mu_g}\right) l_{*,0} \|x_1 - x_2\|  + \frac{l_{g,1}l_{f,0}}{\mu_g^2} l_{g,2} l_{*,0} \|x_1 - x_2\|.
    \end{align*}
    Thus, 
    \begin{align*}
        l_{F,1} &\le l_{*,0} \left(l_{f,1} + \frac{l_{f,0} l_{g,2} + l_{g,1}^2}{\mu_g} + \frac{l_{f,0}l_{g,1}l_{g,2}}{\mu_g^2} \right) \\
        &\le l_{*,0} \left(l_{f,1} + \frac{l_{g,1}^2}{\mu_g} + \frac{2l_{f,0}l_{g,1}l_{g,2}}{\mu_g^2} \right),
    \end{align*}
    where in the last inequality we used $l_{g,1}/\mu_g \ge 1$.
\end{proof}

\subsubsection{Proof of Lemma \ref{lemma:relation_Lagrangian_F}}
We use a short-hand $y^* = y^*(x)$.
\begin{align*}
    \grad_x \mL_\lambda(x,y) &= \grad_x f(x, y) + \lambda (\grad_x g(x,y) - \grad_x g(x, y^*)) \\
    \grad_y \mL_\lambda(x,y) &= \grad_y f(x, y) + \lambda \grad_y g(x,y).
\end{align*}
Check that 
\begin{align}
    \grad F(x) - \grad_x \mL_\lambda(x,y) &= \grad_x f(x,y^*) - \grad_x f(x,y) \nonumber \\
    &\quad - \grad_{xy}^2 g(x,y^*) \grad_{yy}^{2} g(x,y^*)^{-1} \grad_y f(x,y^*) - \lambda (\grad_x g(x,y) - \grad_x g(x, y^*)).  \label{eq:F_minus_gradx_lambda}
\end{align}
We can rearrange terms for $(\grad_x g(x,y) - \grad_x g(x, y^*))$ as the following:
\begin{align}
    \grad_x g(x,y) - \grad_x g(x, y^*) &= \grad_x g(x,y) - \grad_x g(x, y^*) - \grad_{xy} g(x,y^*)^\top (y-y^*) \nonumber \\
    &\quad + \grad_{xy} g(x,y^*)^\top (y-y^*). \label{eq:lambda_gradx_expension}
\end{align}
Note that from the optimality condition for $y^*$, $\grad_{y} g(x,y^*)=0$ and from $\grad_x f(x,y) + \lambda \grad_y g(x,y) = \grad_y \mL(x,y)$, we can express $y - y^*$ as
\begin{align}
    y-y^* &= -\grad_{yy} g(x,y^*)^{-1} (\grad_y g(x,y) - \grad_y g(x,y^*) - \grad_{yy} g(x,y^*) (y-y^*)) \nonumber \\
    &\quad + \frac{1}{\lambda} \grad_{yy} g(x,y^*)^{-1} \left(\grad_y \mL(x,y) - \grad_y f(x,y)\right). \label{eq:lambda_grady_expension}
\end{align}
Plugging \eqref{eq:lambda_gradx_expension} and \eqref{eq:lambda_grady_expension} back to \eqref{eq:F_minus_gradx_lambda}, we have
\begin{align*}
    \grad F(x) - \grad_x \mL_\lambda(x,y) &= (\grad_x f(x,y^*) - \grad_x f(x,y)) - \grad_{xy}^2 g(x,y^*) \grad_{yy}^{2} g(x,y^*)^{-1} (\grad_y f(x,y^*) - \grad_y f(x,y)) \\
    &\quad - \grad_{xy}^2 g(x,y^*) \grad_{yy}^{2} g(x,y^*)^{-1} \grad_y \mL(x,y) \\
    &\quad - \lambda (\grad_x g(x,y) - \grad_x g(x, y^*) - \grad_{xy}^2 g(x,y^*)^\top (y-y^*)) \\
    &\quad + \lambda \grad_{xy}^2 g(x,y^*) \grad_{yy}^2 g(x,y^*)^{-1} (\grad_y g(x,y) - \grad_y g(x,y^*) - \grad_{yy}^2 g(x,y^*) (y-y^*)).
\end{align*}

By the smootheness of $\grad g$ from Assumption \ref{assumption:nice_functions}, we have
\begin{align*}
    \| \grad_y g(x, y) - \grad_y g(x, y^*) - \grad_{yy}^2 g(x, y^*) (y - y^*) \| \le l_{g,2} \|y - y^*\|^2.
\end{align*}
When $\|y - y^*\|$ is too large, the smoothness of $g$ can be more useful:
\begin{align*}
    \| \grad_y g(x, y) - \grad_y g(x, y^*) - \grad_{yy}^2 g(x, y^*) (y - y^*) \| \le 2 l_{g,1} \|y - y^*\|. 
\end{align*}
Similarly, we have
\begin{align*}
    \| \grad_x g(x, y) - \grad_x g(x, y^*) - \grad_{xy}^2 g(x, y^*)^\top (y - y^*) \| \le \min \left(l_{g,2} \|y - y^*\|^2, 2l_{g,1} \|y-y^*\| \right). 
\end{align*}

On the other hand, by smootheness of $f$, we also have
\begin{align*}
    \|\grad_x f(x,y^*) - \grad_x f(x,y) \| \le l_{f,1} \|y - y^*\|, \ \|\grad_y f(x,y^*) - \grad_y f(x,y) \| \le l_{f,1} \|y - y^*\|.  
\end{align*}
We can conclude that
\begin{align*}
    &\| \grad F(x) - \grad_x \mL_\lambda(x,y) + \grad_{xy}^2 g(x,y^*) \grad_{yy}^{2} g(x,y^*)^{-1} \grad_y \mL(x,y) \| \\
    &\qquad \le l_{f,1} (1 + l_{g,1}/\mu_g) \|y - y^*\| + \lambda (1 + l_{g,1}/\mu_g) \|y-y^*\| \min(l_{g,2} \|y-y^*\|, 2l_{g,1}).
\end{align*}
We know that $l_{g,1}/\mu_g \ge 1$ and thus, we have
\begin{align*}
    &\| \grad F(x) - \grad_x \mL_\lambda(x,y) + \grad_{xy}^2 g(x,y^*) \grad_{yy}^{2} g(x,y^*)^{-1} \grad_y \mL(x,y) \| \\
    &\qquad \le 2 (l_{g,1}/\mu_g) \|y-y^*\| \left(l_{f,1} + \lambda \cdot \min(2l_{g,1}, l_{g,2} \|y-y^*\|)\right),
\end{align*}
yielding the lemma.

\subsubsection{Proof of Lemma \ref{lemma:nice_y_star_lagrangian}}
\begin{proof}
    Lipschitzness of $y_{\lambda}^*(x)$ is immediate from Lemma \ref{lemma:y_star_lagrangian_continuity}. By the optimality condition for $\grad y_\lambda^*(x)$, we have
    \begin{align*}
        \grad_y \mL_{\lambda} (x, y_{\lambda}^*(x)) = \grad_y f (x, y_{\lambda}^*(x)) + \lambda \grad_y g (x, y_{\lambda}^*(x)) = 0.
    \end{align*}
    Taking derivative with respect to $x$, we get
    \begin{align*}
        (\grad_{yy}^2 f(x, y_{\lambda}^*(x)) + \lambda \grad_{yy}^2 g(x, y_{\lambda}^*(x))) \grad y_{\lambda}^*(x) = -(\grad_{xy}^2 f(x, y_{\lambda}^*(x)) + \lambda \grad_{xy}^2 g(x, y_{\lambda}^*(x))).
    \end{align*}
    As $\lambda > 2l_{f,1} / \mu_g$, the left-hand side is positive definite with mininum eigenvalue larger than $\lambda \mu_g / 2$, and we have
    \begin{align*}
        \grad y_{\lambda}^*(x) = - \left(\frac{1}{\lambda} \grad_{yy}^2 f(x, y_{\lambda}^*(x)) + \grad_{yy}^2 g(x, y_{\lambda}^*(x)) \right)^{-1} \left( \frac{1}{\lambda} \grad_{xy}^2 f(x, y_{\lambda}^*(x)) + \grad_{xy}^2 g(x, y_{\lambda}^*(x)) \right).
    \end{align*}
    To get the smoothness result, we compare this at $x_1$ and $x_2$, yielding
    \begin{align*}
        \frac{\lambda \mu_g}{2} \|\grad y_{\lambda}^*(x_1) - \grad y_{\lambda}^*(x_2)\| &\le  (l_{f,2} + \lambda l_{g,2}) (\|x_1 - x_2\| + \|y_{\lambda}^*(x_1) - y_{\lambda}^*(x_2)\|) \max_{x \in X} \|\grad y_{\lambda}^*(x)\| \\
        &\quad + (l_{f,2} + \lambda l_{g,2}) (\|x_1 - x_2\| + \|y_{\lambda}^*(x_1) - y_{\lambda}^*(x_2)\|) \\
        &\le (l_{f,2} + \lambda l_{g,2}) (1 + l_{\lambda,0})^2 \|x_1 - x_2\|. 
    \end{align*}
    Arranging this, we get
    \begin{align*}
        \|\grad y_{\lambda}^*(x_1) - \grad y_{\lambda}^*(x_2)\| \le 32 \left(\frac{l_{f,2}}{\lambda} + l_{g,2}\right) \frac{l_{g,1}^2}{\mu_g^3} \|x_1-x_2\|.
    \end{align*}
\end{proof}

\subsubsection{Proof of Lemma \ref{lemma:y_star_contraction}} 
    \begin{proof}
    This is immediate from Lipschitz continuity in Lemma \ref{lemma:y_star_lagrangian_continuity} with sending $\lambda_1 = \lambda_2$ to infinity. 
    \begin{align*}
        \Exs[\|y^*(x_{k+1}) - y^*(x_k)\|^2 | \mathcal{F}_k] &\le l_{*,0}^2 \Exs[\|x_{k+1} - x_k\|^2 | \mathcal{F}_k] \\
        &\le l_{*,0}^2 \xi^2 \alpha_k^2 (\Exs[\|q_k^x\|^2|\mF_k] + \alpha_k^2 \sigma_f^2 + \beta_k^2 \sigma_g^2). 
    \end{align*}
\end{proof}

\subsubsection{Proof of Lemma \ref{lemma:y_star_smoothness_bound}}
\label{appendix:proof:y_star_smoothness_bound}
\begin{proof}
    We can use the smoothness property of $y^*(x)$ as in \cite{chen2021closing}, which is crucial to control the noise variance induced from updating $x$. We can start with the following:
    \begin{align*}
        \vdot{v_k}{y_{k+1}^* - y_k^*} &= \vdot{v_k}{\grad y^*(x_k) (x_{k+1} - x_k)} \\
        &\quad + \vdot{v_k}{y^*(x_{k+1}) - y^*(x_k) - \grad y^*(x_k) (x_{k+1} - x_k)}.
    \end{align*}
    On the first term, taking expectation and using $\vdot{a}{b} \le c \|a\|^2 + \frac{1}{4c} \|b\|^2$,
    \begin{align*}
        \Exs[\vdot{v_k}{\grad y^*(x_k) (x_{k+1} - x_k)} | \mathcal{F}_k] &= -\xi\alpha_k \Exs[\vdot{v_k}{\grad y^*(x_k) q_k^x } | \mathcal{F}_k] \\
        &\le \xi\alpha_k \eta_k \Exs[\|v_k\|^2 | \mathcal{F}_k] + \frac{\xi \alpha_k}{4\eta_k} \Exs[\|\grad y^*(x_k) q_k^x\|^2 | \mathcal{F}_k] \\
        &\le \xi \alpha_k \eta_k \Exs[\|v_k\|^2| \mathcal{F}_k] + \frac{\xi\alpha_k l_{*,0}^2}{4\eta_k} \Exs[\|q_k^x\|^2 | \mathcal{F}_k],
    \end{align*}
    where we used the Lipschitz continuity of $y^*(x)$. For the second term, using smoothness of $y^*(x)$, 
    \begin{align*}
        & \Exs[\vdot{v_k}{y^*(x_{k+1}) - y^*(x_k) - \grad y^*(x_k) (x_{k+1} - x_k)} | \mathcal{F}_k] \\
        &\le \frac{l_{*,1}}{2} \Exs[ \|v_k\| \|x_{k+1}-x_k\|^2 | \mathcal{F}_k] \\
        &\le \frac{l_{*,1}}{4} \Exs\left[ \left(l_{*,1} \|v_k\|^2 + \frac{1}{l_{*,1}} \right) \cdot \|x_{k+1}-x_k\|^2 | \mathcal{F}_k \right] \\
        &\le \frac{l_{*,1}^2 }{4} \Exs[\|v_k\|^2 \cdot \Exs\left[\|x_{k} - x_{k+1}\|^2 | \mathcal{F}_{k}' \right] | \mathcal{F}_k] \\
        &\quad + \frac{\xi^2}{4} \left(\alpha_k^2 \Exs[\|q_k^x\|^2]  + \alpha_k^2 \sigma_f^2 + \beta_k^2 \sigma_g^2 \right),
    \end{align*}
    where $\mathcal{F}_k'$ is a sigma-algebra generated by stochastic noises up to $k^{th}$ iteration and $v_k$. Note that 
    \begin{align*}
        \Exs\left[\|x_{k} - x_{k+1}\|^2 | \mathcal{F}_{k}' \right] &\le \xi^2 \alpha_k^2 \Exs\left[\|q_k^x\|^2 | \mathcal{F}_{k} \right] + \xi^2 (\alpha_k^2 \sigma_f^2 + \beta_k^2 \sigma_g^2),
    \end{align*}
    and from boundedness of $\grad_x f$ and $\grad_x g$ in Assumption \ref{assumption:bounded_grad_x}, we have $\alpha_k \|q_k^x\| \le \alpha_k l_{f,0} + 2 \beta_k l_{g,0}$. With $M_f = l_{f,0}^2 + \sigma_f^2$, $M_g = l_{g,0}^2 + \sigma_g^2$, and $M = \max(M_f, M_g)$, we get
    \begin{align*}
        \Exs\left[\|x_{k} - x_{k+1}\|^2 | \mathcal{F}_{k}' \right] \le 2 \xi^2 (M_f \alpha_k^2 + 2 M_g \beta_k^2) \le 4M\xi^2 l_{*,1}^2 \beta_k^2, 
    \end{align*}
    which yields
    \begin{align*}
        & \Exs[\vdot{v_k}{y^*(x_{k+1}) - y^*(x_k) - \grad y^*(x_k) (x_{k+1} - x_k)} | \mathcal{F}_k] \\
        &\quad \le M \xi^2 l_{*,1}^2 \beta_k^2 \Exs[\|v_k\|^2 | \mathcal{F}_k]  
        + \frac{\xi^2}{4} \left(\alpha_k^2 \Exs[\|q_k^x\|^2 | \mathcal{F}_k ]  + \alpha_k^2 \sigma_f^2 + \beta_k^2 \sigma_g^2 \right). 
    \end{align*}
    Combining all, we obtain the desired result. 
\end{proof}

\subsubsection{Proof of Lemma \ref{lemma:y_star_lambda_vdot_bound}}
\label{appendix:proof:y_star_lambda_smoothness_bound}
\begin{proof}
    We can start with the following decomposition:
    \begin{align*}
        \vdot{v_k}{y_{\lambda_{k+1}}^*(x_{k+1}) - y_{\lambda_k}^* (x_k)} &= \vdot{v_k}{y_{\lambda_{k+1}}^*(x_{k+1}) - y_{\lambda_k}^* (x_{k+1})} \\
        &\quad + \vdot{v_k}{\grad y_{\lambda_k}^*(x_k) (x_{k+1} - x_k)} \\
        &\quad + \vdot{v_k}{y_{\lambda_k}^*(x_{k}) - y_{\lambda_k}^*(x_k) - \grad y_{\lambda_{k}}^*(x_k) (x_{k+1} - x_k)}.
    \end{align*}
    For the second and third terms, we can apply the smoothness of $y_{\lambda}(x)$ similarly in the proof in \ref{appendix:proof:y_star_smoothness_bound}. 
    
    On the first term, taking expectation and using $\vdot{a}{b} \le c \|a\|^2 + \frac{1}{4c} \|b\|^2$,
    \begin{align*}
        \Exs[\vdot{v_k}{y_{\lambda_{k+1}}^*(x_{k+1}) - y_{\lambda_k}^* (x_{k+1})} | \mathcal{F}_k]
        &\le c \Exs[\|v_k\|^2] + \frac{1}{4c} \Exs[\|y_{\lambda_{k+1}}^*(x_{k+1}) - y_{\lambda_k}^* (x_{k+1})\|^2 ] \\
        &\le c \Exs[\|v_k\|^2] + \frac{1}{c} \frac{\delta_k^2}{\lambda_k^2 \lambda_{k+1}^2} \frac{l_{f,0}^2}{\mu_g^2},
    \end{align*}
    where we applied Lemma \ref{lemma:y_star_lagrangian_continuity}. Take $c = \frac{\delta_k}{\lambda_k}$, getting
    \begin{align*}
        \Exs[\vdot{v_k}{y_{\lambda_{k+1}}^*(x_{k+1}) - y_{\lambda_k}^* (x_{k+1})} | \mathcal{F}_k]
        &\le \frac{\delta_k}{\lambda_k} \Exs[\|v_k\|^2] + \frac{l_{f,0}^2 \delta_k}{\mu_g^2 \lambda_k^3}. 
    \end{align*}
    Adding this with bounds on other two terms, we get the lemma.
\end{proof}

\subsection{Proofs for Auxiliary Lemmas with Momentum}

\subsubsection{Proof of Lemma \ref{lemma:y_star_contraction_momentum}}
Due to Lipschitz continuity of $y^*(x)$, we have
\begin{align*}
    \Exs[\|y^*(x_{k+1}) - y^*(x_k)\|^2] &\le l_{*,0}^2 \Exs[\|x_{k+1} - x_k\|^2] \\
    &\le \xi^2 \alpha_k^2 l_{*,0}^2 \Exs[\|q_k^x + \tilde{e}_k^x\|^2] \le 2\xi^2 \alpha_k^2 l_{*,0}^2 (\Exs[\|q_k^x\|^2] + \Exs[\|\tilde{e}_k^x\|^2]).
\end{align*}

\subsubsection{Proof of Lemma \ref{lemma:y_star_lambda_contraction_momentum}}
Using Lemma \ref{lemma:y_star_lagrangian_continuity}, we have
\begin{align*}
    \Exs[\|y_{\lambda_{k+1}}^*(x_{k+1}) - y_{\lambda_{k}}^*(x_k)\|^2] &\le \frac{8\delta_k^2}{\lambda_k^2\lambda_{k+1}^2} + 2 l_{\lambda,0}^2 \Exs[\|x_{k+1} - x_k\|^2] \\
    &\le 4 \xi^2 \alpha_k^2 l_{*,0}^2 (\Exs[\|q_k^x\|^2] + \Exs[\|\tilde{e}_k^x\|^2]) + \frac{8\delta_k^2}{\lambda_k^4}.
\end{align*}

\subsubsection{Proof of Lemma \ref{lemma:descent_stochastic_noise_yz}}
\label{appendix:proof_stochastic_noise_yz}
We can start with unfolding the expression for $\Exs[\|\tilde{e}_{k+1}^z\|^2]$. 
\begin{align*}
    \Exs[\| \tilde{e}_{k+1}^z\|^2] &= \Exs[\| \tilde{h}_{z}^{k+1} - q_{k+1}^z\|^2] \\
    &= \Exs[\| \grad_y g(x_{k+1}, z_{k+1}; \phi_{z}^{k+1}) + (1 - \eta_{k+1}) (\tilde{h}_{z}^{k} - \grad_y g(x_k, z_k; \phi_z^{k+1})) - q_{k+1}^z\|^2] \\
    &= \Exs[\| (1 - \eta_{k+1}) \tilde{e}_k^z + \grad_y g(x_{k+1}, z_{k+1}; \phi_{z}^{k+1}) \\
    &\qquad + (1 - \eta_{k+1}) (q_k^z - \grad_y g(x_k, z_k; \phi_z^{k+1})) - q_{k+1}^z\|^2] \\
    &= (1 - \eta_{k+1})^2 \Exs[\| \tilde{e}_k^z\|^2 ] + \Exs[\|\eta_k (\grad_y g(x_{k+1}, z_{k+1}; \phi_{z}^{k+1}) - q_{k+1}^z) \\
    &\qquad + (1-\eta_{k+1}) (\grad_y g(x_{k+1}, z_{k+1}; \phi_{z}^{k+1}) - \grad_y g(x_k, z_k; \phi_{z}^{k+1}) + q_{k}^z - q_{k+1}^z) \|^2].
\end{align*}
In the last equality, we used
\begin{align*}
    \Exs[ \Exs[ \vdot{\tilde{e}_k^{z}}{\grad_y g(x_{k+1}, z_{k+1}; \phi_{z}^{k+1}) - q_{k+1}^z} | \mathcal{F}_{k+1}] ] &= 0, \\
    \Exs[ \Exs[ \vdot{\tilde{e}_k^{z}}{\grad_y g(x_{k}, z_{k}; \phi_{z}^{k+1}) - q_k^z} | \mathcal{F}_{k+1}] ] &= 0.
\end{align*}
Also note that 
\begin{align*}
    \Exs[\|\grad_y g(x_{k+1}, z_{k+1}; \phi_{z}^{k+1}) - q_{k+1}^z\|^2] \le \sigma_g^2,
\end{align*}
from the variance boundedness (Assumption \ref{assumption:gradient_variance}). We also observe that
\begin{align*}
    \Exs[\|\grad_y g(x_{k+1}, z_{k+1};\phi_{z}^{k+1}) - \grad_y g(x_k, z_k; \phi_{z}^{k+1}) \|^2] &\le l_{g,1}^2 (\|x_{k+1} - x_k\|^2 + \|z_{k+1} - z_k\|^2) \\
    &= l_{g,1}^2 (\xi^2 \alpha_k^2 \|q_k^x + \tilde{e}_k^x\|^2 + \gamma_k^2 \|q_k^z + \tilde{e}_k^z\|^2), 
\end{align*}
due to Assumption \ref{assumption:nice_stochastic_fg}. The same inequality holds for $q_{k+1}^z - q_k^z$:
\begin{align*}
    \Exs[\|q_{k+1}^z - q_k^z \|^2] \le l_{g,1}^2 (\xi^2 \alpha_k^2 \|q_k^x + \tilde{e}_k^x\|^2 + \gamma_k^2 \|q_k^z + \tilde{e}_k^z\|^2). 
\end{align*}
Now we plug these inequalities and using $\|a+b\|^2 \le 2(\|a\|^2 + \|b\|^2)$ multiple times, we have
\begin{align*}
    \Exs[\|\tilde{e}_{k+1}^z\|^2] &\le (1 - \eta_{k+1})^2 (1 + 8 l_{g,1}^2 \gamma_k^2) \Exs[\|\tilde{e}_k^z\|^2] + 2 \eta_{k+1}^2 \sigma_g^2 \\
    &\qquad + 8l_{g,1}^2 (1-\eta_{k+1})^2 \left( \xi^2\alpha_k^2 \Exs[\|q_k^x\|^2] + \xi^2\alpha_k^2 \Exs[\|\tilde{e}_k^x\|^2] + \gamma_k^2 \Exs[\|q_k^z\|^2] \right).
\end{align*}

Similarly, we can repeat similar steps for $\tilde{e}_{k+1}^y$. To simplify the notation, with slight abuse in notation, we let $q_k^y(\zeta, \phi) := \grad_y f(x_{k}, y_{k};\zeta) + \lambda_{k} \grad_y g (x_k, y_k; \phi)$. Note that $q_k^y = \Exs[q_k^y(\zeta, \phi)]$. Then we can get a similar bound for $\Exs[\| \tilde{e}_{k+1}^y \|^2]$:
\begin{align*}
    \Exs[\| \tilde{e}_{k+1}^y\|^2 ] &\le (1-\eta_{k+1})^2 \Exs[\| \tilde{e}_{k}^y\|^2 ] + 2 \eta_{k+1}^2 \Exs[\| q_{k+1}^y(\zeta_{y}^{k+1}, \phi_{y}^{k+1}) - q_{k+1}^y \|^2] \\
    &\qquad + 2(1-\eta_{k+1})^2 \Exs[\| (q_{k+1}^y (\zeta_{y}^{k+1}, \phi_{y}^{k+1}) - q_k^y (\zeta_{y}^{k+1}, \phi_{y}^{k+1})) + (q_k^y - q_{k+1}^y)\|^2 ].
\end{align*}
Using the variance bound similarly, we have
\begin{align*}
    \Exs[\| q_{k+1}^y(\zeta_{y}^{k+1}, \phi_{y}^{k+1}) - q_{k+1}^y \|^2] \le \sigma_f^2 + \lambda_{k+1}^2 \sigma_g^2.
\end{align*}
Then, we unfold the last term such that
\begin{align*}
    \Exs[&\| (q_{k+1}^y (\zeta_{y}^{k+1}, \phi_{y}^{k+1}) - q_k^y (\zeta_{y}^{k+1}, \phi_{y}^{k+1})) + (q_k^y - q_{k+1}^y)\|^2 ] \\
    &= \Exs[\| (\grad_y f(x_{k+1}, y_{k+1}; \zeta_y^{k+1}) - \grad_y f(x_{k}, y_{k}; \zeta_y^{k+1}) + \grad_y f(x_{k}, y_k) - \grad_y f(x_{k+1}, y_{k+1})) \\
    &\quad + \lambda_k (\grad_y g(x_{k+1}, y_{k+1}; \phi_y^{k+1}) - \grad_y g(x_k, y_k; \phi_y^{k+1}) + \grad_y g(x_k, y_k) - \grad_y g(x_{k+1}, y_{k+1})) \\
    &\quad + \delta_k (\grad_y g(x_{k+1},y_{k+1}; \phi_{y}^{k+1}) - \grad_y g(x_{k+1}, y_{k+1}) + \grad_y g(x_k,y_k) - \grad_y g(x_k, y_k; \phi_{y}^{k+1}) ) \|^2] \\
    &\le 12 (l_{f,1}^2  + l_{g,1}^2 \lambda_k^2) (\|x_{k+1} - x_k\|^2 + \|y_{k+1} - y_k\|^2) + 12 \delta_k^2 \sigma_g^2 \\
    &\le 24 (l_{f,1}^2 \alpha_k^2 + l_{g,1}^2 \beta_k^2) (\xi^2 \|q_k^x\|^2 + \xi^2 \|\tilde{e}_k^x\|^2 + \|q_k^y\|^2 + \|\tilde{e}_k^y\|^2) + 12 \delta_k^2 \sigma_g^2.
\end{align*}
We note that we set $\lambda_k \ge 2l_{f,1} / \mu_g$, and thus $l_{f,1} \le \lambda_k l_{g,1}$. In total, we get
\begin{align*}
    \Exs[\| \tilde{e}_{k+1}^y\|^2 ] &\le (1-\eta_{k+1})^2 (1 + 96 l_{g,1}^2 \beta_k^2 ) \Exs[\| \tilde{e}_{k}^y\|^2 ] + 2 \eta_{k+1}^2 (\sigma_f^2 + \lambda_{k+1}^2 \sigma_g^2) + 24 \delta_k^2 \sigma_g^2 \\
    &\qquad + 96 (1-\eta_{k+1})^2 l_{g,1}^2 \beta_k^2 (\xi^2 \|q_k^x\|^2 + \xi^2 \|\tilde{e}_k^x\|^2 + \|q_k^y\|^2).
\end{align*}

\subsubsection{Proof of Lemma \ref{lemma:descent_stochastic_noise_x}}
\label{appendix:proof_stochastic_noise_x}
Similarly to the case for $\|\tilde{e}_{k+1}^y\|^2$, let us define $q_k^x (\zeta, \phi) := \grad_x f(x_k, y_{k+1}; \zeta) + \lambda_k (\grad_x g(x_k, y_{k+1}; \phi) - \grad_x g(x_k, z_{k+1}; \phi))$. We note that $\zeta_x^k, \phi_{x}^k$ are sampled after $y_{k+1}, z_{k+1}$ is updated but before $x_k$ is updated. Hence, 
\begin{align*}
    \Exs[\Exs[\vdot{\tilde{e}_k^x}{q_{k+1}^x(\zeta_x^{k+1}, \phi_{x}^{k+1}) - q_{k+1}^x} | \mathcal{F}_{k+1}']] = 0, \\
    \Exs[\Exs[\vdot{\tilde{e}_k^x}{q_{k}^x(\zeta_x^{k+1}, \phi_{x}^{k+1}) - q_{k}^x} | \mathcal{F}_{k+1}']] = 0.
\end{align*}
Thus, following similar procedure, we have
\begin{align*}
    \Exs[\| \tilde{e}_{k+1}^x\|^2] &= \Exs[\| q_{k+1}^x(\zeta_x^{k+1}, \phi_{x}^{k+1}) + (1-\eta_{k+1}) (q_k^x + \tilde{e}_k^x - q_{k}^x (\zeta_x^{k+1}, \phi_{x}^{k+1}) ) - q_{k+1}^x \|^2] \\
    &\le (1 - \eta_{k+1})^2 \Exs[\| \tilde{e}_k^x\|^2 ] + 2\eta_k^2 \Exs[\|q_{k+1}^x (\zeta_x^{k+1}, \phi_{x}^{k+1}) - q_{k+1}^x\|^2] \\
    &\qquad + 2(1-\eta_{k+1})^2 \Exs[\| (q_{k+1}^x (\zeta_{x}^{k+1}, \phi_{x}^{k+1}) - q_k^x (\zeta_{x}^{k+1}, \phi_{x}^{k+1})) + (q_k^x - q_{k+1}^x)\|^2 ].
\end{align*}
Note that
\begin{align*}
    \Exs[&\|q_{k+1}^x (\zeta_x^{k+1}, \phi_{x}^{k+1}) - q_{k+1}^x\|^2] \\
    &= \Exs[\|(\grad_x f(x_{k+1}, y_{k+2}; \zeta_x^{k+1}) - \grad_x f(x_{k+1}, y_{k+2})) \\
    &\quad + \lambda_k (\grad_x g(x_{k+1}, y_{k+2}; \phi_x^{k+1}) - \grad_x g(x_{k+1}, y_{k+2})) + \lambda_k (\grad_x g(x_{k+1}, z_{k+2}; \phi_x^{k+1}) - \grad_x g(x_{k+1}, z_{k+2}))\|^2] \\
    &\le 3 (\sigma_f^2 + \lambda_k^2 \sigma_g^2). 
\end{align*}
Finally, we have
\begin{align*}
    \Exs[&\| (q_{k+1}^x (\zeta_{x}^{k+1}, \phi_{x}^{k+1}) - q_k^x (\zeta_{x}^{k+1}, \phi_{x}^{k+1})) + (q_k^x - q_{k+1}^x)\|^2 ] \\
    &= \Exs[\| (\grad_x f(x_{k+1}, y_{k+2}; \zeta_x^{k+1}) - \grad_x f(x_{k}, y_{k+1}; \zeta_x^{k+1}) + \grad_y f(x_{k}, y_{k+1}) - \grad_y f(x_{k+1}, y_{k+2})) \\
    &\quad + \lambda_k (\grad_x g(x_{k+1}, y_{k+2}; \phi_x^{k+1}) - \grad_x g(x_k, y_{k+1}; \phi_x^{k+1}) + \grad_x g(x_k, y_{k+1}) - \grad_x g(x_{k+1}, y_{k+2})) \\
    &\quad + \lambda_k (\grad_x g(x_{k+1}, z_{k+2}; \phi_x^{k+1}) - \grad_x g(x_k, z_{k+1}; \phi_x^{k+1}) + \grad_x g(x_k, z_{k+1}) - \grad_x g(x_{k+1}, z_{k+2})) \\
    &\quad + \delta_k (\grad_y g(x_{k+1},y_{k+2}; \phi_{x}^{k+1}) - \grad_y g(x_{k+1}, y_{k+2}) + \grad_x g(x_k,y_{k+1}) - \grad_x g(x_k, y_{k+1}; \phi_{x}^{k+1}) ) \\
    &\quad + \delta_k (\grad_x g(x_{k+1},z_{k+2}; \phi_{x}^{k+1}) - \grad_x g(x_{k+1}, z_{k+2}) + \grad_x g(x_k,z_{k+1}) - \grad_x g(x_k, z_{k+1}; \phi_{x}^{k+1}) ) \|^2].
\end{align*}
Using Cauchy-Schwartz inequality, we get
\begin{align*}
    \Exs[&\| (q_{k+1}^x (\zeta_{x}^{k+1}, \phi_{x}^{k+1}) - q_k^x (\zeta_{x}^{k+1}, \phi_{x}^{k+1})) + (q_k^x - q_{k+1}^x)\|^2 ] \\
    &\le 30(l_{f,1}^2 + l_{g,1}^2 \lambda_k^2) (\|x_{k+1} - x_k\|^2 + \|y_{k+2} - y_{k+1}\|^2 + \|z_{k+2} - z_{k+1}\|^2) + 40 \delta_k^2 \sigma_g^2 \\
    &\le 120 l_{g,1}^2 \lambda_k^2 (\xi^2 \alpha_k^2 (\|q_k^x\|^2 + \|\tilde{e}_k^x\|^2) + \alpha_{k+1}^2 (\|q_k^y\|^2 + \|\tilde{e}_k^y\|^2) + \gamma_{k+1}^2 (\|q_k^z\|^2 + \|\tilde{e}_k^z\|^2)) + 40 \delta_k^2 \sigma_g^2.
\end{align*}
Combining all, we obtain the result.

\end{appendices}

\end{document}